\documentclass{gDEA2e}

\usepackage{amsfonts,amsmath,amssymb}

\usepackage{graphicx}
\usepackage{dcolumn}
\usepackage{bm}
\usepackage{graphicx}
\usepackage{float}
\restylefloat{figure}
\numberwithin{equation}{section}

\def\assumptionname{\scshape Assumption}
\newtheorem{assumption}{\assumptionname}[chapter]

\begin{document}
\markboth{E. Braverman and A. Rodkina}{Stabilization with a stochastic control}
\title{{\itshape Stochastic control 
stabilizing unstable or chaotic maps}}

\author{Elena Braverman$^{\rm a}$$^{\ast}$\thanks{$^\ast$Corresponding author. Email:
maelena@ucalgary.ca}
and  Alexandra Rodkina$^{\rm b}$ \\\vspace{6pt}
$^{\rm a}${\em{Department of Mathematics and Statistics, University of Calgary,\\
2500 University Drive N.W., Calgary, AB T2N 1N4, Canada}} \\
$^{\rm b}${\em{Department of Mathematics,
The University of the West Indies,\\ Mona Campus, Kingston, Jamaica}}}

\maketitle

\begin{abstract}
The paper considers a stabilizing stochastic control which can be applied to a variety of unstable and even 
chaotic maps. Compared  to previous methods introducing control by noise, we relax assumptions on the class of maps, 
as well as consider a wider range of parameters for the same maps. This approach allows to stabilize unstable and  chaotic maps by noise.
The interplay between the map properties and the allowed neighbourhood where a solution can start to be stabilized
is explored: as instability of the original map increases, the interval of allowed initial conditions narrows.
A directed stochastic control aiming at getting to the target neighbourhood almost sure is combined 
with a controlling noise. Simulations illustrate that for a variety of problems, an appropriate bounded noise can stabilize an unstable positive equilibrium, without a limitation on the initial value. 

\end{abstract}


\begin{classcode} 
39A50; 37H10; 93D15; 39A30
\end{classcode}\bigskip

\section{Introduction}

Significant interest to discrete models is stimulated by complicated types of behavior exhibited even by
simple maps. 

The idea of stabilizing an unstable equilibrium of differential equations by noise originates from the  work of R. Khasminskii on stochastic stability \cite{Hasmin}, see also the most recent edition of the monograph \cite{Khas}.   
In 1983, the possibility to stabilize a linear system by noise was demonstrated 
in the paper of L. Arnold {\em et al}~\cite{Arnold}.   
This approach  was expanded  and developed in later works: 
for stochastic differential and functional differential equations see, e.g. \cite{AM,AMR1,Carab,Mao}, 
and for stochastic difference equations in \cite{ABR,AMR06}.

Kolmogorov's  Law of Large Numbers was  applied in the proof of stability of the zero equilibrium for  
linear and nonlinear stochastic non-homogeneous equations  in \cite{BR0,BR1}, and for systems with  
square nonlinearities  in  \cite{Medv}. 
This approach was  originated by H. Kesten (see e.g.  \cite{FK} for a linear model, 
and \cite {K} for convergence in probability). 

Introduction of the stochasticity into the population dynamics description, as well as into  controls, is quite 
a natural part of a model design due to many reasons. 
This includes an extrinsic noise, which may be described 
by a state-independent, or additive, stochastic perturbation. 
The implementation of control cannot be done precisely, leading to a state-dependent, or
multiplicative, stochastic perturbation. 
The influence of stochasticity on population survival, 
chaos control and eventual cyclic behavior was investigated in \cite{BR2c,BravRodk2}, different types of control 
which include stochastic perturbations were considered for the scalar case in \cite{BKR,BR}, see 
the recent papers \cite{Bash1,Bash2,Faure} and references therein for systems. 

In \cite{BRAllee} it was shown that in certain cases, an additive non-decaying stochastic perturbation  
could completely diminish the Alee effect. 
We were able to prove this result,  using the fact that for a given sequence $(\xi_n)_{n\in \mathbb N}$ of independent identically 
distributed random variables, for a subinterval of their support and any  
number $\bar J\in \mathbb N$, there exists a random number $\mathcal N$, 
starting from which $\bar J$ random variables,  in a row,  take values in this interval with probability 1.
The influence of stochastic perturbations on the Allee effect in 
discrete systems was also recently studied in  \cite{Elaydi2016,Roth}.

Concerning the noise component of controls in \cite{BKR,BR},  
it slightly reduced the range of control parameters where stabilization is guaranteed. 
Unlike \cite{BKR,BR}, in the present paper noise is an important {\em stabilizing factor}:
an otherwise unstable positive equilibrium becomes stable after introducing an unstructured noise,
once the initial value is close enough to this equilibrium. The idea is inspired by both physical 
and biological models. As examples, we consider population dynamics models: logistic and Ricker.
The fact that stochastic perturbations can stabilize an unstable equilibrium was
discovered for physical models in the 1950s, following the well-known example of
the pendulum of Kapica \cite{Kapica}. Here we generalize it to a wide range of population
dynamics models, in a discrete setting.

We consider a deterministic difference equation
\begin{equation}
\label{eq:detmainintr}
z_{n+1}={\rm f}(z_n), \quad n\in \mathbb N, \quad  z_0>0,
\end{equation}
where ${\rm f}:[0, \infty)\to :[0, \infty)$ is continuous and has a unique positive fixed point $K>0$. 
We suppose that the  equilibrium $K$ is unstable, and moreover, in  some  interval $(K-u, K+u)$, 
the function $\rm f$ can be represented for some $u, q, C, \kappa >0$ as
\begin{equation}
\label{repr:rmfintr}
{\rm f}(z)=K-(1+q)(z-K)+\phi(z-K), \quad |\phi(z-K)|\le C|z-K|^{1+\kappa}.
\end{equation}
 
Our aim is to construct a stochastic control which stabilizes the equilibrium $K$ for  
initial values $x_0\in (K-u, K+u)$ with a given probability.  
Since equation \eqref{eq:detmainintr} is applied to population dynamics models, 
we consider only bounded noises, which are supposed to be mutually independent. 
Compared to previous methods introducing control by noise, we relax assumptions on the class of maps and
consider a wider range of parameters for the function ${\rm f}$. In particular, the original map can be chaotic. However,
the stronger the map instability is, the smaller is the allowed neighbourhood where a solution can start.
With chaotic maps and values coming occasionally as close to the equilibrium as required, with a probability close to one
the solution eventually at least once enters the target neighbourhood. However, the situation changes if there is an attractive cycle or
other orbit separated from the equilibrium point. To attain the target domain, we apply a directed stochastic control 
which is stopped once a solution is in the required neighbourhood.

The main type of control includes noise only, it can be considered as cost-free, if such a noise is natural.
However, first an additional control should push a solution into a smaller target interval $(K-\delta, K+\delta)$. 
Note that to reach this interval, various control methods can be applied, for example, Prediction 
Based Control \cite{BKR} or non-stochastic Target Oriented Control \cite{Dattani,TPC}. 
As stochasticity is an intrinsic part of a control, we incorporate noise at this preliminary control stage as well.
Prediction Based and Proportional Feedback controls \cite{BKR,BR} included either multiplicative or additive 
noise. We introduce a Directed Walk Control (DWC) 
\begin{equation}
\label{def:DWIntr}
z_{n+1}={\rm f}(z_n) -\alpha_j (1+ \ell_c \chi_{n+1}) \frac{{\rm f}(z_n)-K}{|{\rm f}(z_n)-K|}+ \ell \zeta_{n+1}, ~z_n \in J^{(j)},~
j=0, \dots, k-1,
\end{equation}
acting on a finite number of nested intervals $J^{(j)} \subset J^{(j-1)}$, starting from some interval $J^{(0)}=(K-u, K+u)$,  until the solution $z$, after a 
finite number of steps which can be explicitly computed, reaches  the target interval  $J^{(k)}=(K-\delta, 
K+\delta)$.  Here $(\chi_n)_{n\in \mathbb N}$ and $(\zeta_n)_{n\in \mathbb N}$ are sequences of independent bounded noises.
In addition to its ability to reach a target interval, which any of the
mentioned methods can, DWC
includes, unlike \cite{BKR,BR}, both a multiplicative and an additive noise, which is more natural in the setting of the model.
In addition, the total number of steps to reach the target interval is estimated (see the appendix),
the step where the solution enters it is random, however, the maximal cost can be evaluated analytically. 
This is, generally, not an easy task, see, for example, \cite{Dattani}.

As soon as, at a random moment $\tau=\inf\{n\in \mathbb N: z_n\in (K-\delta, K+\delta)\}$,  the solution $z$ reaches 
the target interval $(K-\delta, K+\delta)$,   the second type of control  is launched, which eventually brings $z$ 
to the equilibrium.  This is the main control, which is described in Section~\ref{sec:expest} and is based on the application of the Kolmogorov's  Law of 
Large Numbers. In this paper we call this method the Multiplicative Noise Control (MNC). 
This control is genuinely  stochastic,  and application of such a control can be viewed as an example of 
stabilization by noise. After applying MNC, the equation takes the form
\begin{equation}
\label{eq:MNCintr}
z_{n+1}={\rm f}(z_n)+\sigma \xi_{n+1}(z_n-K), \quad n\ge \tau, \quad z_\tau\in (K-\delta, K+\delta) .
\end{equation}
For  $\rm f$ satisfying \eqref{repr:rmfintr}, suggested in \eqref{eq:MNCintr} control   stabilizes the equilibrium $K$ if 
\begin{equation}
\label{cond:Eqs}
- \lambda := \mathbb E\ln |1+q-\sigma \xi_{n+1} |<0.
\end{equation}
For quite a variety of   independent and identically distributed random variables  $\xi_n$ 
inequality \eqref{cond:Eqs} is valid  when $q$ and $\sigma$ are small enough, and  $\sigma^2>2q,$ see e.g. \cite{Medv,K}. 
In Section~\ref{sec:exsim} we present calculations of $\lambda$ in the two cases: $\xi_n$
are continuous uniformly distributed on $[-1, 1]$ random variables and $\xi_n$ are Bernoulli distributed random  variables  taking the values of  1 and -1 with equal probabilities $p=1-p=0.5$.   
It appears that   $q$ and $\sigma$ do not need to  be  too small to ensure fulfillment of condition \eqref{cond:Eqs}. In particular, for Bernoulli distributed $\xi_n$, 
condition \eqref{cond:Eqs} holds  for any $q\ge 0$, when $\sigma$ is chosen appropriately (see \eqref{calc:Evarbernpar} in Section \ref{sec:exsim}). The case of $q=1$, $\sigma =2.1$  is illustrated by simulation, see Fig.~\ref{figure7}. 

The  main result of the paper states that, under conditions \eqref{repr:rmfintr} and \eqref{cond:Eqs}, for any 
$\gamma\in (0, 1)$,   we can find a $\delta>0$ such that after applying  control \eqref{def:DWIntr} 
for $z_0\in (K-u, K+u)$  followed by control \eqref {eq:MNCintr}, 
for $z_\tau\in (K-\delta, K+\delta)$, we have $\lim\limits_{n\to \infty}z_n=K$ with probability greater than $1-\gamma$.

Calculations in  Sections~\ref{sec:expest}  show that the radius  $\delta$ of the interval for the initial value 
$x_0$ can be extremely small.  To enhance  this situation, in Section~\ref{sec:comb}, we suggest the combined  method, where the first control pushes the solution into a larger interval $(K-\beta, K+\beta)$, where 
$(K-\delta, K+\delta)\subset(K-\beta, K+\beta)\subset (K-u, K+u)$,  and the solution is returned 
to  $(K-\beta, K+\beta)$ each time it gets out. 
Applying an alternative stabilization by noise  approach, 
see Lemma~\ref{lem:barprob}, we prove that the solution still reaches  $(K-\delta, K+\delta)$ in a.s. finite time.   
Thus we consider a control with switching.  Since the moment when the solution enters the designated interval is 
random, the switching happens at random moments,  
which  brings some  extra complications and makes proofs quite technical (see Section~\ref{sec:expest}). 

The main differences between the results of our paper and some previous research, for example, 
\cite{Medv}, can be outlined as follows. 
\begin{enumerate}
\item
We combine stabilization by noise with other control methods. However, at the final stage we have only stabilization by 
a state-dependent noise.
\item
In general we don't assume that $q$ is small. 
However, the larger is $q$, the smaller the initial neighbourhood should be. We note that in numerical runs wider 
neighbourhoods are applied than rigorously predicted theoretically, the estimates are {\em sufficient only}.
Also, $\kappa$ in \eqref{repr:rmfintr}  can be any positive number, while $\kappa=1$ necessarily in \cite{Medv}.
\item
In \cite{Medv}, switching between two equilibrium points is controlled: the noise brings the zero equilibrium which loses stability, to become stable again, in a small neighbourhood of the equation parameter.
In our examples, we consider mostly positive equilibrium points which lose stability (giving rise to stable cycles and eventually chaos), and become stable under controls with stochastic perturbations. The maps which are stabilized can even be chaotic.
\end{enumerate}

The paper is organized as follows. 
After describing all relevant definitions, assumptions and notations for equation~\eqref{eq:detmainintr} 
in Section~\ref{sec2}, we state the main results of the paper justifying the possibility to stabilize a map by noise 
in Section~\ref{sec:mainres}. However, such stabilization is stipulated by the choice of the initial point in the close proximity 
of the initial point to the unstable equilibrium, especially for chaotic maps. To alleviate this requirement, Section~\ref{sec:comb} considers a combination of another method applied at some finite and initially evaluated number of steps, and the noise control. This allows to start 
MNC from a significantly wider interval, in fact practically from any positive value. 
In Section~\ref{sec:exsim} we consider several types of equations, in particular Ricker and logistic,
as well as another type for which none of the results developed in \cite{Medv} can be applied. 
In numerical simulations, we use either uniformly or Bernoulli distributed noise. 
Finally, Section~\ref{sec:discussion} contains a summary and a discussion, and describes possible developments of the present research.
Most of involved proofs of the main results are given in Appendix A. We introduce  Directed Walks Control (DWC) as an additional method, all the details are postponed till Appendix B.


\section{Preliminaries}
\label{sec2}

Let  $(\Omega, {\mathcal{F}},  (\mathcal{F}_n)_{n \in \mathbb{N}}, {\mathbb{P}})$ be  a complete filtered probability space.  
In the paper we deal with  three  sequences of random variables $(\xi_n)_{n\in \mathbb N}$, $(\zeta_n)_{n\in \mathbb N}$, and $(\chi_n)_{n\in \mathbb N}$, each of them satisfying  the following assumption.

\begin{assumption}
\label{as:chixi}
Each sequence $(\zeta_n)_{n\in \mathbb N}$, $(\chi_n)_{n\in \mathbb N}$ and $(\xi_n)_{n\in \mathbb N}$ consists of identically  distributed  random variables, and
$\zeta_n$, $\chi_n$ and $\xi_n$ are mutually independent random variables satisfying $|\zeta_n| \leq 1$, $|\chi_n| \leq 1$, $|\xi_n| \leq 1$, $n\in \mathbb N$.
\end{assumption}
 
We use the standard abbreviation ``a.s." for the wordings ``almost sure" or ``almost surely" 
with respect to the fixed probability measure $\mathbb P$  throughout the text. 
A detailed discussion of stochastic concepts and notation may be found, for example, in \cite{Shiryaev96}. 

Everywhere below, for each $t\in [0, \infty)$, we denote by $[t]$ the integer part of $t$, $\mathbb N_0:=\mathbb N\cup \{0\}$, and, for $\delta, u\in (0, \infty)$, $0<\delta<u$,  
\begin{equation}
\label{def:IduK}
I_\delta:=(-\delta, \, \delta), ~ I_u:=(-u, \, u), 
~I_{u, K}:=(K-u, K+u),~ I_{\delta, K}:=(K-\delta, K+\delta).
\end{equation}
In the paper we consider identically distributed random variables, $\xi_n$. Sometimes, when we deal with their 
probabilities, $\mathbb P\{\xi_n>a\}$, or expectations,   $\mathbb E \ln (1+q-\sigma \xi_n)$,  the index $n$ could be omitted. 

\subsection{Kolmogorov's  Law of Large Numbers and two more lemmas}
\label{subsec:Kolm}

In this section we formulate the Kolmogorov's  Law of Large Numbers, 
see i.e. \cite {Shiryaev96}, and two lemmas. The proof of Lemma \ref{lem:aux} is straightforward and is thus omitted. The proof of Lemma \ref{lem:barprob} can be found, e.g in \cite{BKR} or \cite{BRAllee}.
\begin{theorem}[(Kolmogorov)]
\label{thm:Kolm}
  Let $(v_{n})_{n\in\ {\mathbb N}}$ be the sequence of independent random
  variables with $\theta_n^2=Var (v_n)<\infty$. Let $S_n=v_1+\dots
  +v_n$ and an increasing sequence of $b_n>0$ be such that $b_n\uparrow + \infty$ as $n\to
  \infty$  and
  \begin{equation*}
    \sum_{i=1}^{\infty}\frac{\theta_i^2}{b_i^2}<\infty.
    \label{Kolm}
  \end{equation*}
  Then a.s.,
  $\displaystyle   \frac{S_n-{\bf E}S_n}{b_n} \rightarrow 0$.
\end{theorem}

\begin{lemma} 
\label{lem:aux}
  Let $(e_n)_{n\in\ \mathbb N}$, $e_n \in {\mathbb R}$ be an increasing sequence,
$\displaystyle \lim_{n\to\infty} e_n = +\infty$,
and
  $(b_n)_{n\in\ \mathbb N}$ be a sequence satisfying $\displaystyle \sum_{j=0}^{\infty}b_j < \infty$.
  Then
  $\displaystyle
  \lim_{n\to\infty}\frac{1}{e_n}\sum_{j=0}^{n} b_je_j = 0$.
\end{lemma}

\begin{lemma} (see \cite{BKR,BRAllee})
\label{lem:barprob}
Let $(\xi_n)_{n\in\ \mathbb N} $ be a sequence of independent identically distributed random variables such that $\mathbb P \left\{\xi_n\in (a,b]\right\}=p_1\in (0, 1)$ for some interval $(a,b]$, $a<b$, and each $n\in\mathbb{N}_0$. Then for each  nonrandom $\bar J \in {\mathbb N}$, the probability
\begin{multline*}
\mathbb{P}\left[\text{There exists } \mathcal N=\mathcal N(\bar J)<\infty\right.\\
\left. \text{such that }   \xi_{\mathcal N}\in (a,b], \, \xi_{\mathcal N+1}\in [a,b), \dots, 
\xi_{\mathcal N+\bar J}\in (a,b] \right]=1.
\end{multline*}
\end{lemma}

\subsection{Transformation of the equation}
\label{subsec:rmffxi}

In this section we formulate the main assumptions about  $\rm f$, and transform the equation 
\begin{equation}
\label{eq:detmain}
z_{n+1}={\rm f}(z_n), \quad n\in \mathbb N, \quad z_0>0,
\end{equation}
moving the equilibrium $K$ to 0, as well as introduce some notations and comments on the MNC (Multiplicative Noise 
Control)  method. 
\begin{assumption}
\label{as:fu}
Let $ \mathrm f:[0, \infty)\to [0, \infty)$ be a continuous function which 
has a unique positive fixed point $K>0$, $ (z-K)({\rm f}(z)-K)<0,~z\neq K$,
\begin{equation}
\label{repr:rmf}
{\rm f}(z)=K-(1+q)(z-K)+\phi(z-K),
\end{equation}
whenever $z\in I_{u, K}$, and for some fixed positive $u, q, C, \kappa$,
\begin{equation}
\label{eq:phi}
|\phi (x)|\le C|x|^{1+\kappa}, \quad \forall x\in  I_u.
\end{equation}
\end{assumption}

Denote  
\begin{equation}
\label{def:f}
x:=z-K, \,\, f(x)=\mathrm f(z-K)-K,\,\,  \text{so }\,\,  x\in I_u\,\,  \text{when} \,\,  z\in I_{u, K}.
\end{equation}
When $\rm f$ satisfies Assumptions \ref{as:fu},  transformation \eqref{def:f} moves the equilibrium $K$  to zero. 
For the function $f$ defined by \eqref{def:f}, we have   
\begin{equation}
\label{def:simplf}
\begin{split}
&f(x)<0, \quad x\in (0, u), \quad f(x)>0, \quad  x\in (-u, 0),\\
&f(x)=-(1+q)x+\phi(x), \quad x\in I_u,
\end{split}
\end{equation}
where $\phi$ satisfies \eqref{eq:phi}. 
So now the  equation
\begin{equation}
\label{eq:detf}
x_{n+1}=f(x_n) 
\end{equation}
has the unstable zero equilibrium. 
To stabilize the equilibrium, we  introduce  the noise term $\sigma x_n\xi_n$ into the 
right-hand-side of \eqref {eq:detf}. This  leads us to the stochastic difference equation
\begin{equation}
\label{eq:stoch1}
x_{n+1}= f(x_n)+\sigma \xi_{n+1}x_n.
\end{equation}
For the sequence $(\xi_n)_{n\in \mathbb N}$ satisfying  Assumption \ref{as:chixi} we  set
\begin{equation}
\label{def:Theta}
\Theta_n:=1+q-\sigma \xi_{n}
\end{equation}
and note that $\Theta_n$ are also mutually independent  and identically distributed random variables.
By \eqref{def:simplf},  for  $x_n\in I_u$, we have 
\[
 f(x_n)+\sigma \xi_{n+1}x_n=-\Theta_{n+1} x_n+\phi(x_n),
 \]
 so, when $x_n\in I_u$, equation  \eqref{eq:stoch1} can be written as
 \begin{equation}
\label{eq:stoch21}
x_{n+1}=-(1+q-\sigma \xi_{n+1}) x_n+\phi(x_n)=-\Theta_{n+1} x_n+\phi(x_n).
\end{equation}
For $\Theta_n$ defined as in \eqref{def:Theta} we set
\begin{equation}
\label{def:vi}
v_n:=\ln|\Theta_n|,
\end{equation}
and note that $v_n$ are also mutually independent  and identically distributed. 
Also, both $\Theta_n$ and $v_n$,  are bounded:
 \begin{equation}
\label{rel:boundThetav}
\begin{split}
&1+q-\sigma\le \Theta_n\le 1+q+\sigma, \quad \ln|1+q-\sigma|\le v_n\le \ln (1+q+\sigma),\\
& |\Theta_n|\le 1+q+\sigma=:\bar \Theta, \quad |v_n|\le \max\left\{ \left| \ln|1+q-\sigma| \right|, \ln (1+q+\sigma)\right\} =: \bar v.
\end{split}
\end{equation}

In this paper, our main assumption about the values of the coefficient $q$ and of the noise intensity $\sigma $ is 
the following.
\begin{assumption}
\label{as:lambda}
Let Assumption \ref{as:chixi}  hold and $\Theta_n$ be defined as in \eqref{def:Theta}. 
There exists   $\lambda>0$ such that
\begin{equation}
\label{as:Thetalambda}
\mathbf E v_n=  \mathbf E\ln |\Theta_n| = :-\lambda<0.
\end{equation}
\end{assumption}
\begin{remark}
\label{rem:Eln}
As was mentioned in Introduction, Assumption \ref{as:lambda} is fulfilled for many common
independent identically distributed $\xi_n$  when $q$ and $\sigma$ are small enough and $\sigma^2>2q$. 
In Section~\ref{sec:exsim}, we derive $\mathbf E v_n$  for two particular distributions of $\xi_n$, uniform continuous and  Bernoulli. Obtained formulae show that Assumption \ref{as:lambda} is fulfilled for not so small $q$ and $\sigma$.  
Actually in the case of Bernoulli distribution for each $q>0$ we can find $\sigma$ such that Assumption \ref{as:lambda}  holds. Computer simulations of these cases are also provided in Section~\ref{sec:exsim}.
\end{remark}
Whenever \eqref{as:Thetalambda} holds, the application of  the Kolmogorov's  Law of Large Numbers, i.e. Theorem \ref{thm:Kolm} with $v_i:=\ln|\Theta_i|$, $ b_i:=i$,   gives that, a.s.,  
\begin{equation}
\label{rel:KLL}
\frac 1n \sum_{i=1}^n v_i\to -\lambda, \quad \text{as}\,\, n\to \infty.
\end{equation}

Relation \eqref {rel:KLL} implies that $\forall \varepsilon\in (0,\lambda)$, there exists a random 
$\mathcal N_1=\mathcal N_1(\omega, \varepsilon)$ such that $\forall n\ge \mathcal N_1$ we have, a.s.
\begin{equation}
\label{ineq:updown}
-(\lambda+\varepsilon)n<\sum_{i=1}^n v_i\le -(\lambda-\varepsilon)n.
\end{equation}

\section{Main results}
\label{sec:mainres}

In this section we present  three results about a.s. convergence of a solution $z$  to the equilibrium after application of various stochastic controls.

The first result refers to application of only MNC,  when either a solution starts at $I_{\delta, K}$ or an  arbitrary control method brings it into $I_{\delta, K}$ at some a.s. finite random moment $\tau$.

The second result shows that DWC method can actually bring the solution from $I_{u, K}$ to $I_{\delta, K}$ using  a.s. finite number of steps $\tau$.


The third result is a combination of DWC and MNC methods,  when instead of a small neighbourhood $I_{\delta, K}$ we start applying MNC when a solution reaches a much bigger interval $I_{\beta, K}$.

\subsection {MNC only}
\label{subsec:mainse}


Consider an a.s. finite random variable  $\tau$ which takes non-negative  integer values and satisfies  the following assumption.

\begin{assumption}
\label{as:tau1}
Assume that  the random variable $\tau: \Omega \to \mathbb N_0$ is a.s. finite and independent of $\xi_i$ for all $i\in \mathbb N$.
\end{assumption}
Assumption \ref {as:tau1} implies that $\tau$ is also independent of $\Theta_i$, defined by \eqref{def:Theta},  and $v_i=\ln |\Theta_i|$, for all $i\in \mathbb N$.

Assume that  a solution either starts at $I_{\delta, K}$ or an  arbitrary control method brings it into $I_{\delta, K}$. Recall that  MNC  has the form
\begin{equation}
\label{eq:stochmainK}
z_{n+1}= {\rm f}(z_n)+\sigma \xi_{n+1}(z_n-K), \quad z_\tau\in I_{\delta,K},  \quad n\ge \tau.
\end{equation}

\begin{theorem}
\label{thm:main1}
Let   $\gamma \in (0, 1)$, Assumptions~\ref{as:chixi}, \ref{as:fu} and \ref{as:lambda} hold and 
an a.s. finite random moment $\tau$ satisfy Assumption \ref{as:tau1}. Then there exist positive constants 
$\varsigma$, $\eta$ and $\delta$, 
such that for any solution $z_n$  to \eqref{eq:stochmainK} we have 
\begin{enumerate}
\item [(i)] there exists a set $\Omega_\gamma$, $\mathbb P \Omega_\gamma>1-\gamma$,  such that, on $\Omega_\gamma$,
\begin{enumerate}
\item 
$|z_m-K|\le \eta e^{-\varsigma (m-\tau)}, \quad \mbox{for all} \quad m\ge \tau$,
\item $
z_m\to K, \quad \text{as} \quad m\to \infty;
$
\end{enumerate}
\item [(ii)] there exists a set $\Omega^{[1]}_\gamma$, $\mathbb P \Omega^{[1]}_\gamma>1-\gamma$, and a nonrandom $n_0\in \mathbb N$,  such that, on $\Omega^{[1]}_\gamma$,
\[
|z_m-K|\le \eta e^{-\varsigma (m-n_0)}, \quad \forall m\ge n_0.
\]
\end{enumerate}
\end{theorem}

The proof of Theorem \ref{thm:main1} is given in Appendix A, Section \ref{subsec:Prmain1}.


\subsection{DWC reaches $I_{\delta, K}$ in a.s. finite time}
Now consider the case when a solution $z$ 
with $z_0\in I_{u, K}$,  is brought to $I_{\delta,K}$ by DWC method, introduced in \eqref{def:DWIntr} (see details in the  Appendix B). 
Let $\tau$ be  the  moment when  the  solution $z_n$ for the first time  reaches the interval $I_{\delta, K}$,  
\begin{equation}
\label{def:tau}
\tau=\inf\{n: z_n\in I_{\delta, K} \}.
\end{equation}

Assume that for constants $C$, $u$ and $\kappa$ from the Assumption  \ref{as:fu} the following condition holds
\begin{equation}
\label {cond:uCkappa}
Cu^\kappa<\frac{1}{2} \, .
\end{equation}
\begin{theorem}
\label{thm:main2}
Let $\gamma\in (0, 1)$, Assumptions~\ref{as:chixi}, \ref{as:fu}, \ref{as:lambda}  and condition \eqref {cond:uCkappa} 
hold. Then  there exist positive constants $\varsigma$, 
$\eta$, $\delta$ (the same as in Theorem \eqref{thm:main1}), and parameters of  DWC \eqref{def:DWIntr}, such that the moment $\tau$ defined in \eqref{def:tau} satisfies Assumption \ref{as:tau1},  and therefore the statement of Theorem \ref{thm:main1} holds.
\end{theorem}
The proof of Theorem \ref{thm:main2} is given in Appendix B, Section \ref{subsec:Prmain2}.

\subsection{A combined method}
\label{sec:comb}

Let assumptions of Theorem \ref{thm:main1} hold and  $u>\delta$ be the same as in Theorem \ref{thm:main1}.


In this section we combine DWC  \eqref{def:DWIntr} with MNC \eqref {eq:stochmainK},   in order to be able
to start applying   MNC \eqref {eq:stochmainK} in  a  bigger interval  $I_{\beta, K}$,  $u>\beta\gg \delta$ (i.e. with $\beta$ instead of $\delta$).
The size of $I_{\beta, K}$  should be such that for each  $x_n\in I_{\beta, K}$, under further application of $MNC$ at every step, we have
\begin{enumerate}
\item[(i)] $x_{n+1}\in I_{u, K}$;
\item [(ii)] there are $\iota \in (0, 1)$ and $\bar s\in \mathbb N$ such that  $\xi_{n+i}\in (1-\iota,1]$, $ i=1, \dots, \bar s$ implies  $x_{n+\bar s}\in I_{\delta, K}$.
\end{enumerate}
The combined method can be described as follows. For  $x_0\in I_{u, K}$  we apply DWC method \eqref{def:DWIntr} (see Section \ref{sec:DWC} for details) which takes  at most $\bar s_1$ steps to reach $I_{\beta, K}$. As soon the solution is in $I_{\beta}$, we apply MNC method \eqref{eq:stochmainK}  (see Section  \ref{sec:expest} for details). However, at this stage, a solution can get out of 
 $I_{\beta, K}$, but only into $I_{u, K}$  by (i) above. If this happens  we apply  again the DWC method to get the solution back inside  $I_{\beta, K}$ using  at most $\bar s_1$ steps. Note that, once  $x\in I_{u, K}$, also $x_{n+1} \in I_{u, K}$.
To prove that  after a.s. finite numbers of steps the solution gets into $I_{\delta, K}$, we apply Lemma \ref{lem:barprob}. 
As soon as  $x_n\in I_{\delta, K}$, by the results in Section \ref{sec:expest}, a solution becomes asymptotically stable and no more 
DWC is necessary.

In addition to Assumptions \ref{as:chixi}, \ref{as:fu} and \ref{as:lambda} we suppose that
\begin{equation}
\label{ineq:sq}
q<\sigma<2+q;
\end{equation}
\begin{equation}
\label{cond:distr}
\mathbb P\{\xi\in (1-\iota, 1] \}=p_\iota>0, \quad \mbox{for each} \quad \iota\in \left(0, 1-\frac q\sigma \right).
\end{equation}

\begin{remark}
\label{rem:conds}
Condition \eqref {cond:distr} holds in any of the two cases
\\
(a) $\xi_n$ are Bernoulli distributed, $\mathbb P\{\xi=1\}=0.5$,  , $\mathbb P\{\xi=-1\}=0.5$;
\\
(b) $\xi_n$ are continuous random variables with a positive density on $[-1,1]$.

In case (a) we have $p_\iota=0.5$ since 
$
\mathbb P\{\xi>1-\iota\}=\mathbb P\{\xi=1\}=\frac 12.
$

Condition \eqref{ineq:sq} holds, in particular, if $\sigma<2$ and $q<\frac {\sigma^2}2$, since $q <\frac {\sigma^2}2<\sigma$. 
\end{remark}

Define 
\begin{equation}
\label{def:Thetaupdowm}
\underline \Theta_\iota:=1+q-(1-\iota)\sigma, \quad \bar \Theta:=1+q+\sigma.
\end{equation}
By \eqref{cond:distr} we have $1>1-\iota>\frac q\sigma,$ which implies $
 |\underline \Theta_\iota |<1$.
 
\begin{theorem}
\label{thm:main3}
Let $\gamma\in (0, 1)$,  Assumptions~\ref{as:chixi}, \ref{as:fu}, \ref{as:lambda}, conditions \eqref {cond:uCkappa}, \eqref{ineq:sq} and \eqref{cond:distr} hold. Let $\underline \Theta_\iota$ and $\bar \Theta$ be defined as in \eqref{def:Thetaupdowm} and $\beta$ satisfy the inequality
  \begin{equation}
\label{def:beta}
 \beta <\min\left\{\sqrt[\kappa]{\frac{1- |\underline \Theta_\iota |}{C}},  \,\, \frac{u}{\bar\Theta+C}\right\}.
 \end{equation}
 Then there exist  positive constants $\varsigma$, 
$\eta$, $\delta$ (the same as in Theorem \eqref{thm:main1}), and parameters of  DWC \eqref{def:DWIntr}, such that
 a method  which combines application of DWC to $I_{u, K}$ followed by application of MNC to $I_{\beta, K}$ brings solution to  $I_{\delta, K}$  at the moment $\tau$,  which satisfies Assumption \ref{as:tau1},
 and therefore the statement of Theorem \ref{thm:main1} holds.
 \end{theorem}

The proof of Theorem \ref{thm:main3} is given in Appendix B, Section \ref{subsec:Prmain3}.


\section{Examples and computer simulations}
\label{sec:exsim}

In this section we consider two types of functions $\rm f$, logistic  and Ricker, as well as an additional example of an unbounded $\rm f$,
and two types of random variables $\xi$, continuous uniformly distributed on $[-1, 1]$ random variable and Bernoulli distributed random variable
\begin{equation}
\label{def:bernxi}
\xi:=\left\{ \begin{array}{ll} 
1,  & p=0.5;\\
-1,  & 1-p=0.5. \end{array} \right.
 \end{equation}
For $\xi$ defined by \eqref{def:bernxi} by direct calculations we obtain
\begin{equation}
\label{calc:Evarbern}
-\lambda=\mathbf E \ln |\Theta|=\frac 12\ln \left|(1+q)^2-\sigma^2\right|.
\end{equation}
From \eqref {calc:Evarbern} we conclude  that $\lambda>0$ holds if and only if the following estimation is fulfilled:
\begin{equation}
\label{calc:Evarbernpar}
(1+q)^2-1<\sigma^2< (1+q)^2+1.
\end{equation} 
Estimation  \eqref{calc:Evarbernpar} allows parameters $q$ and $\sigma$ to be quite large. 
So even for the values of the equation parameters when the non-controlled equation is chaotic,
and the derivative $|{\rm f}'(K)|=1+q$ with substantial $q>0$, 
the equilibrium $K$ can be stabilized with the help of a noise of type \eqref{def:bernxi}.  
However, such a stabilization comes at a price: the interval $I_{\delta, K}$ becomes incredibly small.  
In this situation it is reasonable to apply the combined method, see Section \ref{sec:comb}, Section \ref{sec:DWC} from Appendix B, and examples  below. 

In some of the examples we use parameters of MNC and DWC methods, which are defined and discussed in Appendices.

\begin{remark}
\label{rem: comb1}
Condition \eqref{calc:Evarbernpar} immediately implies  \eqref {ineq:sq}.
\end{remark}

If $\xi$  is continuously uniformly distributed on $[-1, 1]$, direct calculations lead to
\begin{equation}
\label{calc:Eunif}
-\lambda=\mathbf E \ln |\Theta|=\frac 1{2\sigma}\bigl[ (1+q+\sigma )\ln(1+q+\sigma ) + (\sigma-1-q)\ln|\sigma -1-q|-2\sigma\bigr].
\end{equation}
From \eqref{calc:Eunif} we conclude  that, for big $\sigma$, the first term in the brackets of the right-hand side is dominating, even though the second term can be negative, see examples and simulations below. So only when the derivative ${\rm f}'(K)= -(1+q)$ is  close to -1 (like $q \leq 0.3$), stabilization with the help of a continuous uniformly distributed noise could be achieved.

In the rest of the section we provide examples with  computer simulations which illustrate our theoretical results and also show that,
for smaller $q$, stabilization is possible with MNC only. 


\subsection{The logistic map.}

Consider the truncated logistic equation
\begin{equation}
\label{eq:log}
z_{n+1}={\rm f}(z_n)= \max\{ r z_n(1-z_n), 0 \},
\end{equation}
which has a nonzero equilibrium at $K=1-\frac 1r$, and the coefficients  in \eqref{repr:rmf}  are: $C=r$, $q=r-3$, $\kappa=1$.    
We can show that $u\le \frac 1{4r}$  and coefficients  in \eqref{cond:Lqno2}  are
\[
\underline L=r-2ur-2, \quad \bar q= r+2ur-3.
\]
A similar to \eqref{eq:log} truncation is applied to MNC
\begin{equation}
\label{eq:logstoch} 
z_{n+1}=\max\left\{ {\rm f}(z_n)+\sigma \xi_{n+1}(z_n-K), 0 \right\}.
\end{equation}
For $\iota$ satisfying \eqref{cond:distr}, applying \eqref{def:beta},  we can show that $\beta$ can be calculated as
\begin{equation}
\label{est:beta}
 \beta< 
 \min\left\{\frac {1-|r-2-(1-\iota)\sigma|}{r}, \quad\frac {u}{2r-2+\sigma}\right\}.
\end{equation}

\subsubsection{Logistic map: stabilization by noise only}
\label{sec:by_noise}

In estimation of the proximity of $x_0$ to the equilibrium required to apply MNC and achieve stabilization, 
very small $\delta$ are obtained.
However, in simulations much wider neighbourhoods allow such a control.
First, we consider stabilization by noise (MNC) only. We use \eqref{eq:logstoch},
otherwise, we do not employ any method to keep a solution in a vicinity
of the equilibrium. The initial value is also not assumed to be in a small neighbourhood of $K$.
  
We apply continuous uniformly distributed on $[-1,1]$ noise in Examples~\ref{ex:un1},\ref{ex:un2}.

\begin{example}
\label{ex:un1}
Let $r=3.1$, then  $q=0.1$, $K=0.677$.  Applying \eqref{calc:Eunif}  with $\sigma=0.4$, we get $\lambda\approx -0.0723<0$,  for $\sigma=0.6$ we get $\lambda\approx -0.0405<0$ 
as illustrated on Fig.~\ref{figure1}, upper left for $\sigma=0.4$ and upper right for $\sigma=0.6$.
In both cases there is no stabilization, as predicted. 

If $\sigma=0.8$ then $\lambda \approx 0.012>0$, for 
$\sigma=0.9$, $\lambda \approx 0.051>0$, 
see Fig.~\ref{figure1}, lower left for $\sigma=0.8$ and lower, right for $\sigma=0.9$. 
In both cases we observe stabilization
by MNC, slower for $\sigma=0.8$ and faster for $\sigma=0.9$.
\begin{figure}[ht]
\centering
\includegraphics[height=.16\textheight]{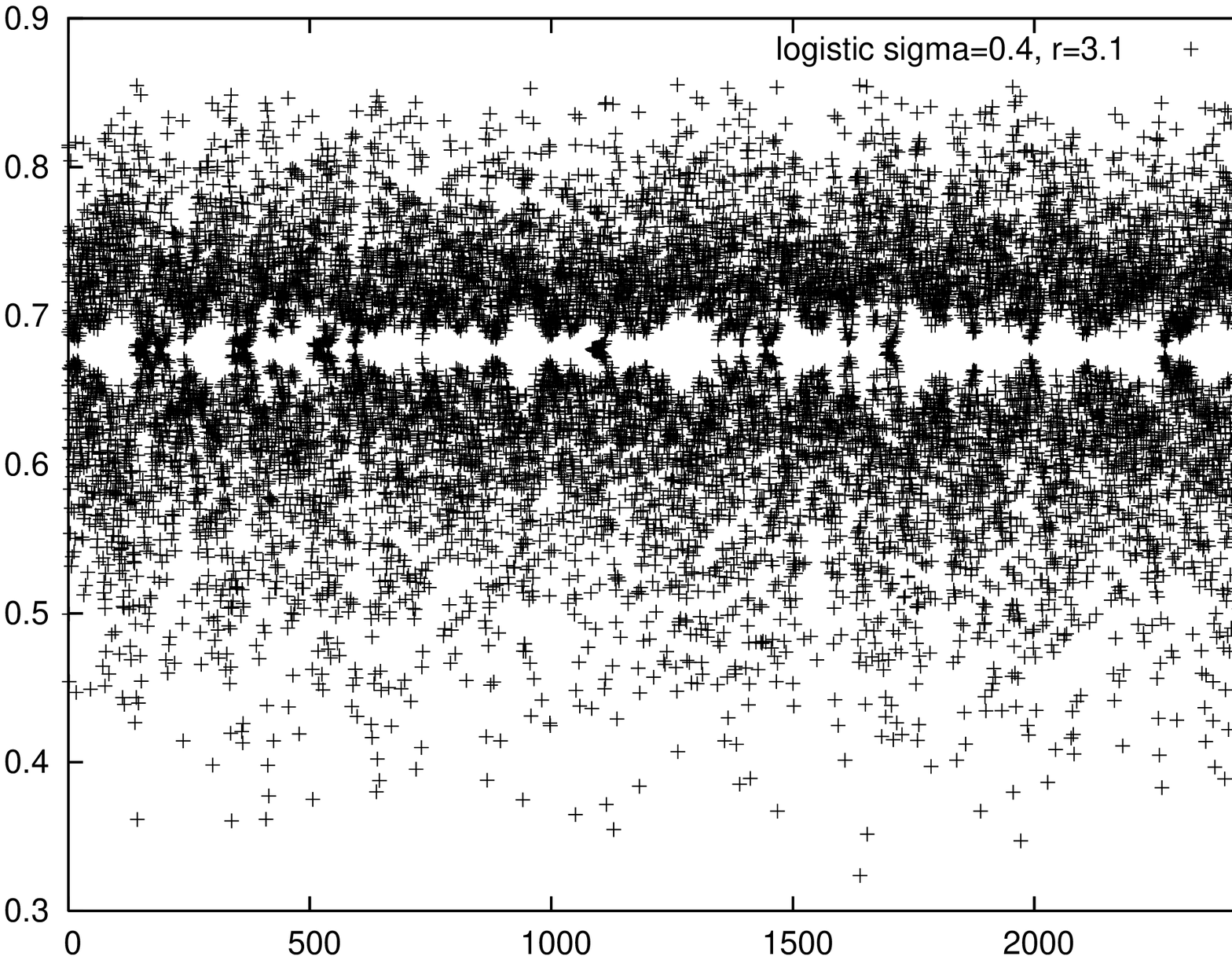}
\hspace{6mm} 
\includegraphics[height=.16\textheight]{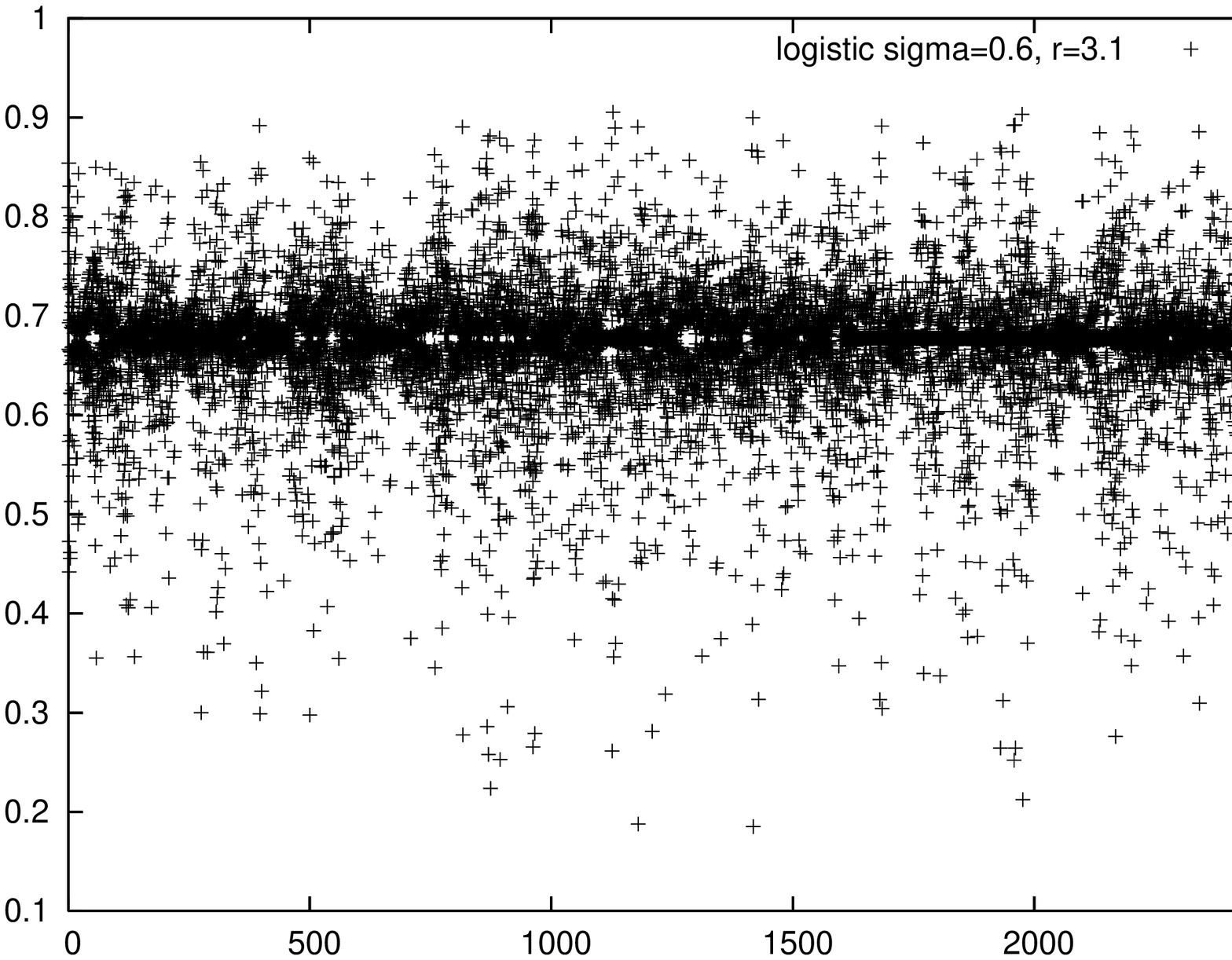}
\vspace{6mm}

\includegraphics[height=.16\textheight]{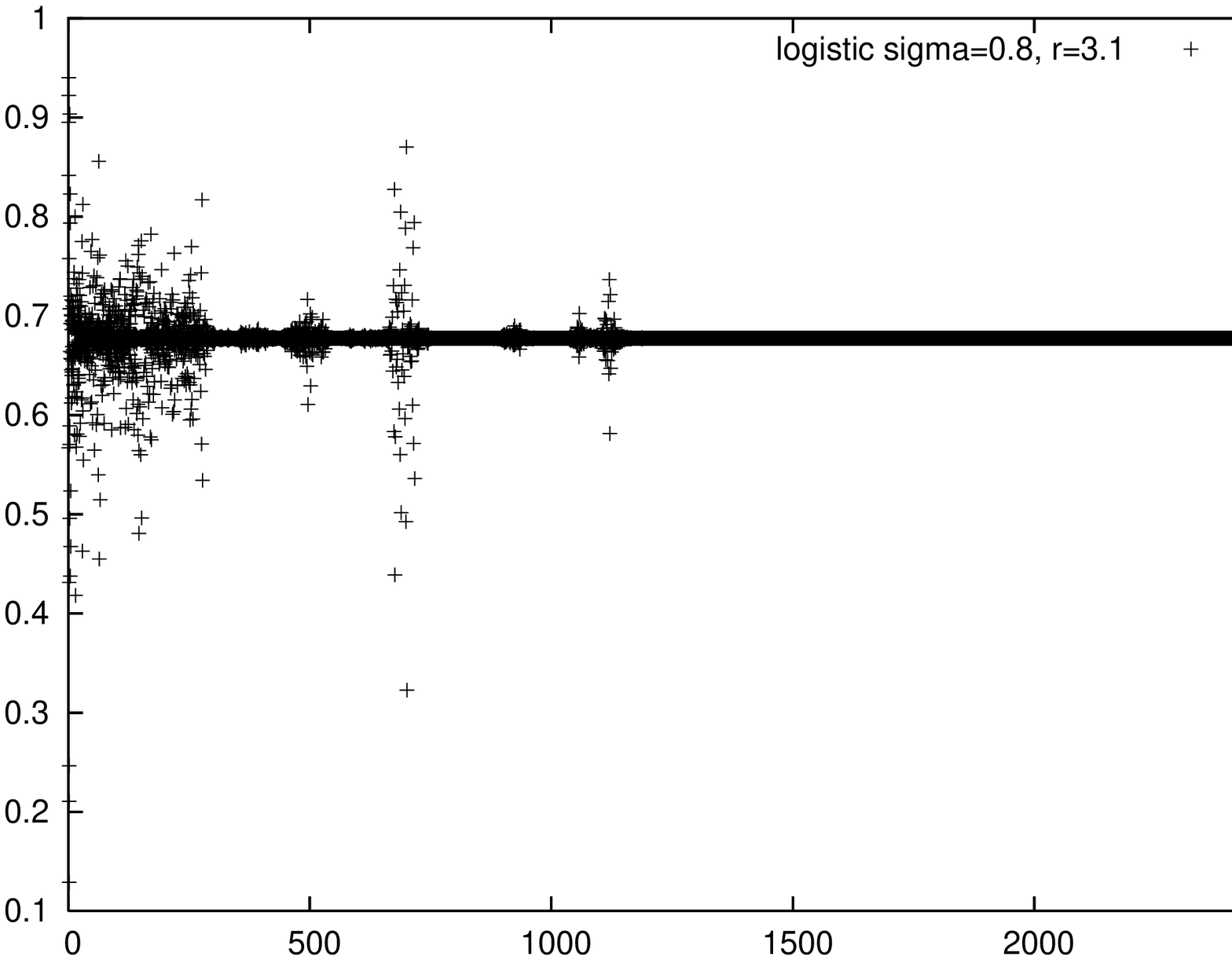}
\hspace{6mm}
\includegraphics[height=.16\textheight]{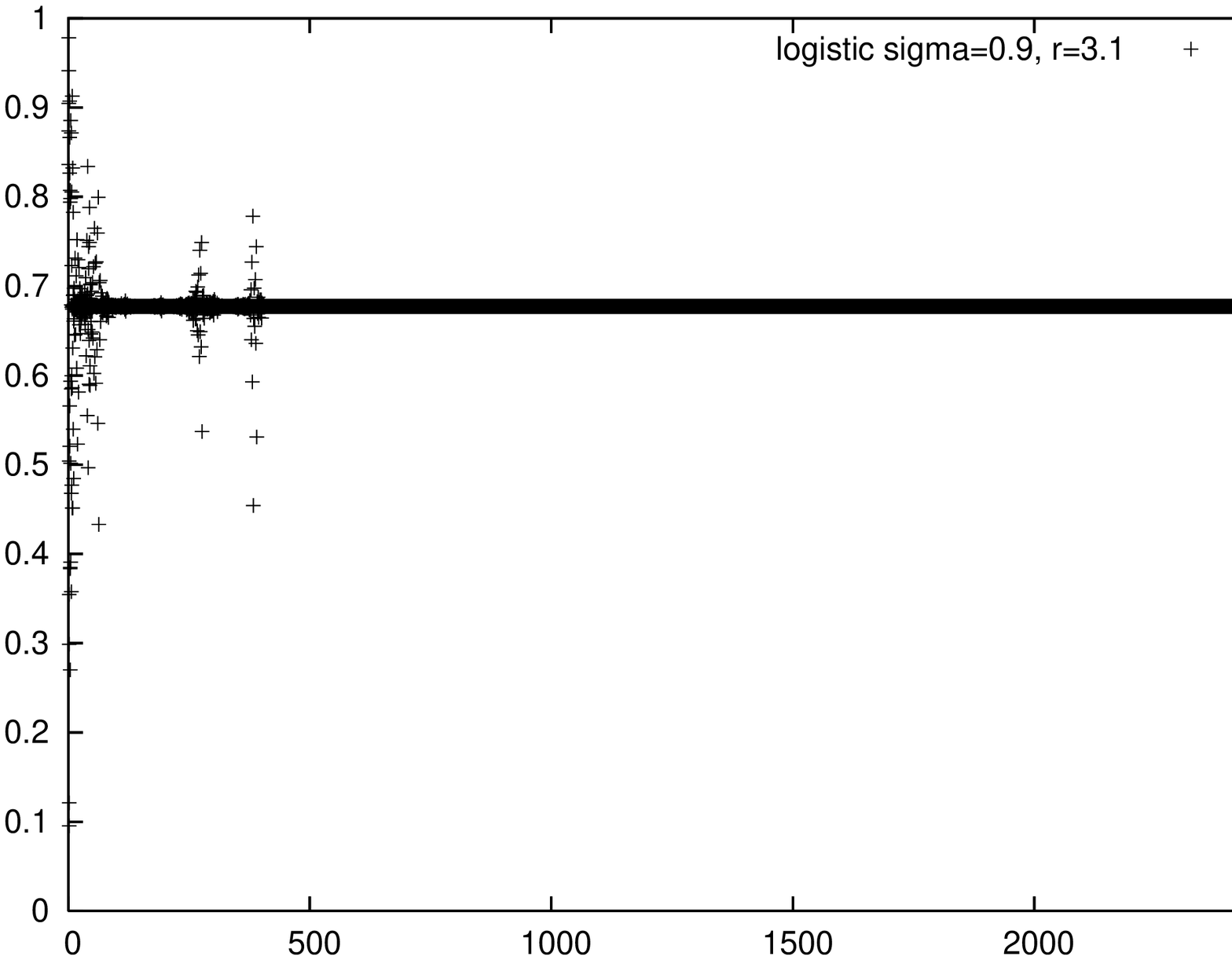}
\vspace{7mm}

\caption{Five runs of the perturbed logistic equation (\protect{\ref{eq:log}}) with $r=3.1$, uniformly distributed in $[-1,1]$ noise,
$x_0=0.4$ and $\sigma=0.4,0.6,0.8,0.9$ (from left to right), respectively.
}
\label{figure1}
\end{figure}

\end{example}

Next, we illustrate the interplay between $q$ and $\sigma$ in MNC: the larger is $q$, the bigger $\sigma$ is required for stabilization.

\begin{example}
\label{ex:un2} 
Let $r=3.3$, then $K=0.7$, $q=0.3$. For $\sigma=0.8$ and $\sigma=1.2$ corresponding $\lambda$'s are negative ($\approx -0.19$ and $\approx -0.05$, respectively), 
so MNC does not stabilize the equation, see Fig.~\ref{figure2}, the upper row.
However, compared to $\sigma=0.8$, see Fig.~\ref{figure2}, upper left, for $\sigma=1.2$, upper right, the solution is more concentrated around $K=0.7$.

For $\sigma=1.3$, we get $\lambda \approx 0.044 >0$,
and MNC is stabilizing, see Fig.~\ref{figure2}, lower left. For $\sigma=1.4$,  $\lambda\approx 0.1245 >0$ guarantees  
convergence (Fig.~\ref{figure2}, lower right).
\begin{figure}[ht]
\centering
\includegraphics[height=.16\textheight]{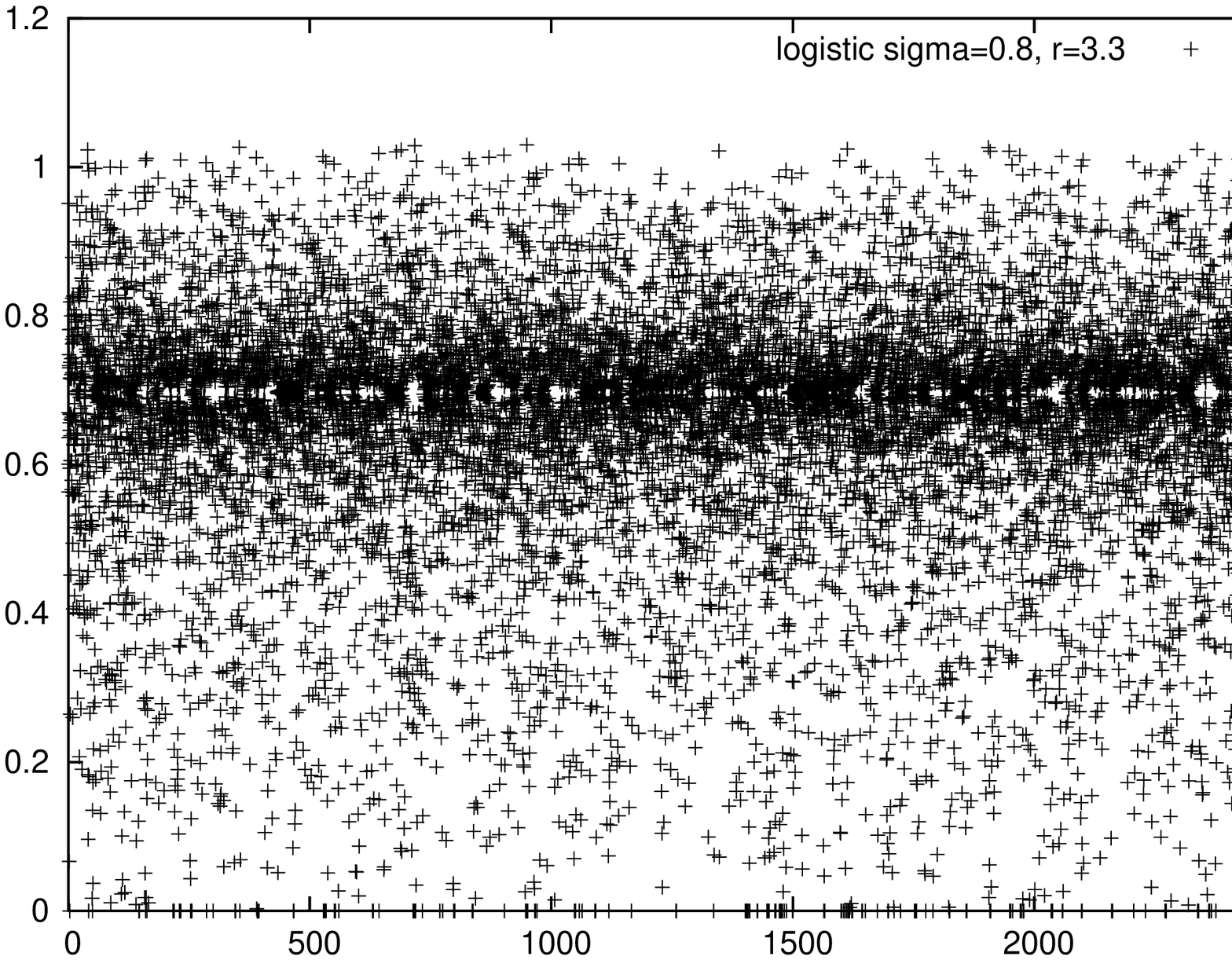}
\hspace{6mm}
\includegraphics[height=.16\textheight]{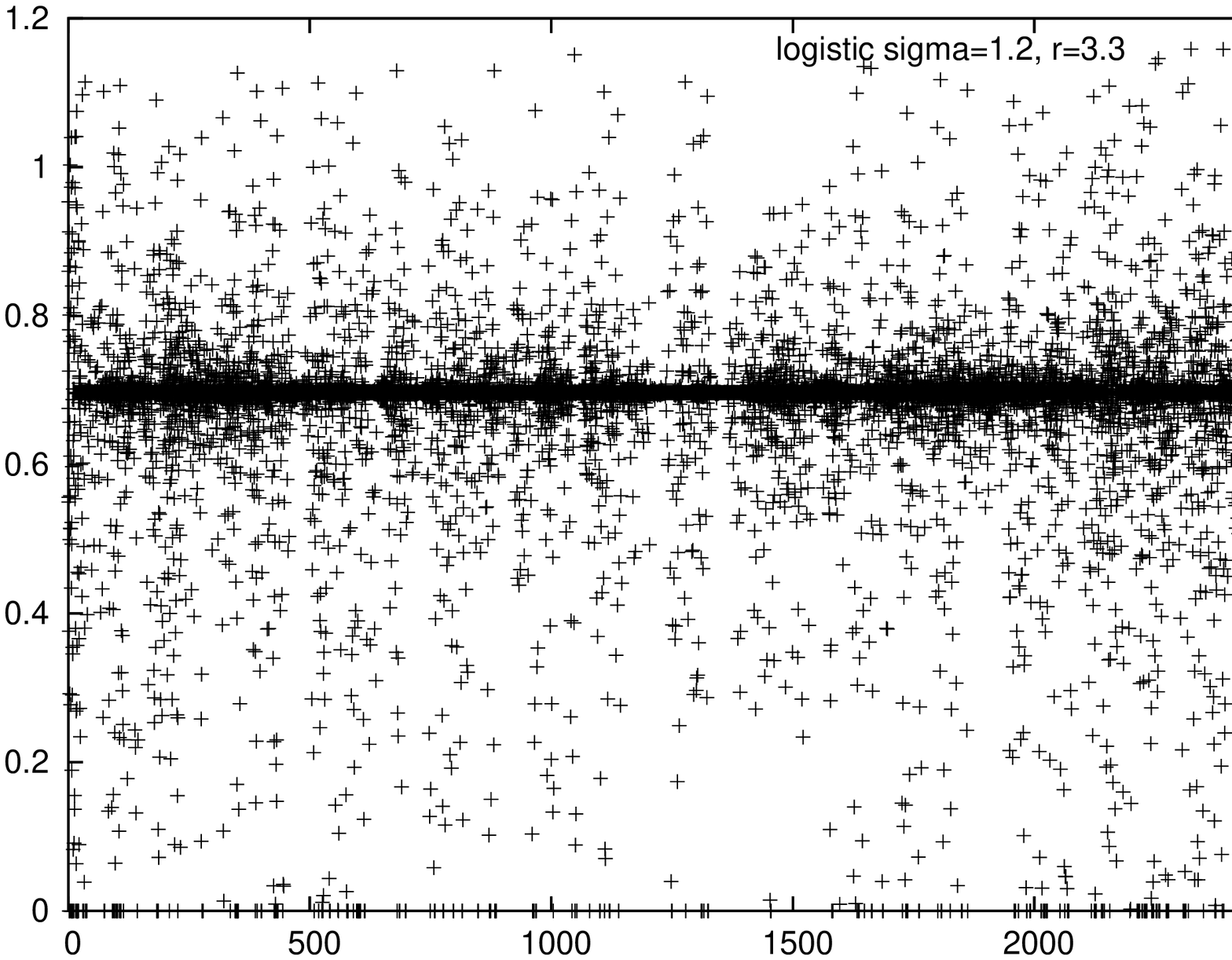}
\vspace{6mm}

\includegraphics[height=.16\textheight]{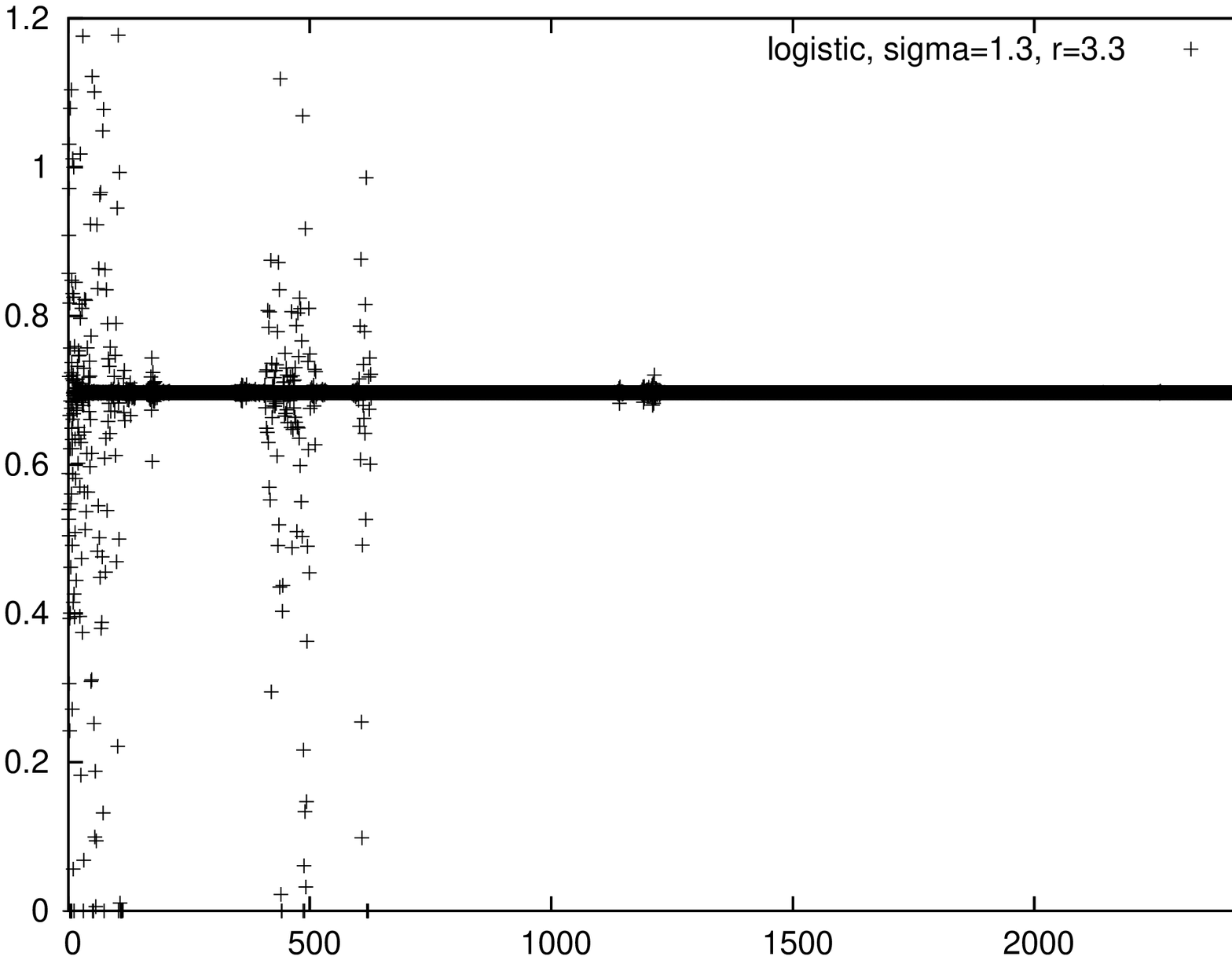}
\hspace{6mm}
\includegraphics[height=.16\textheight]{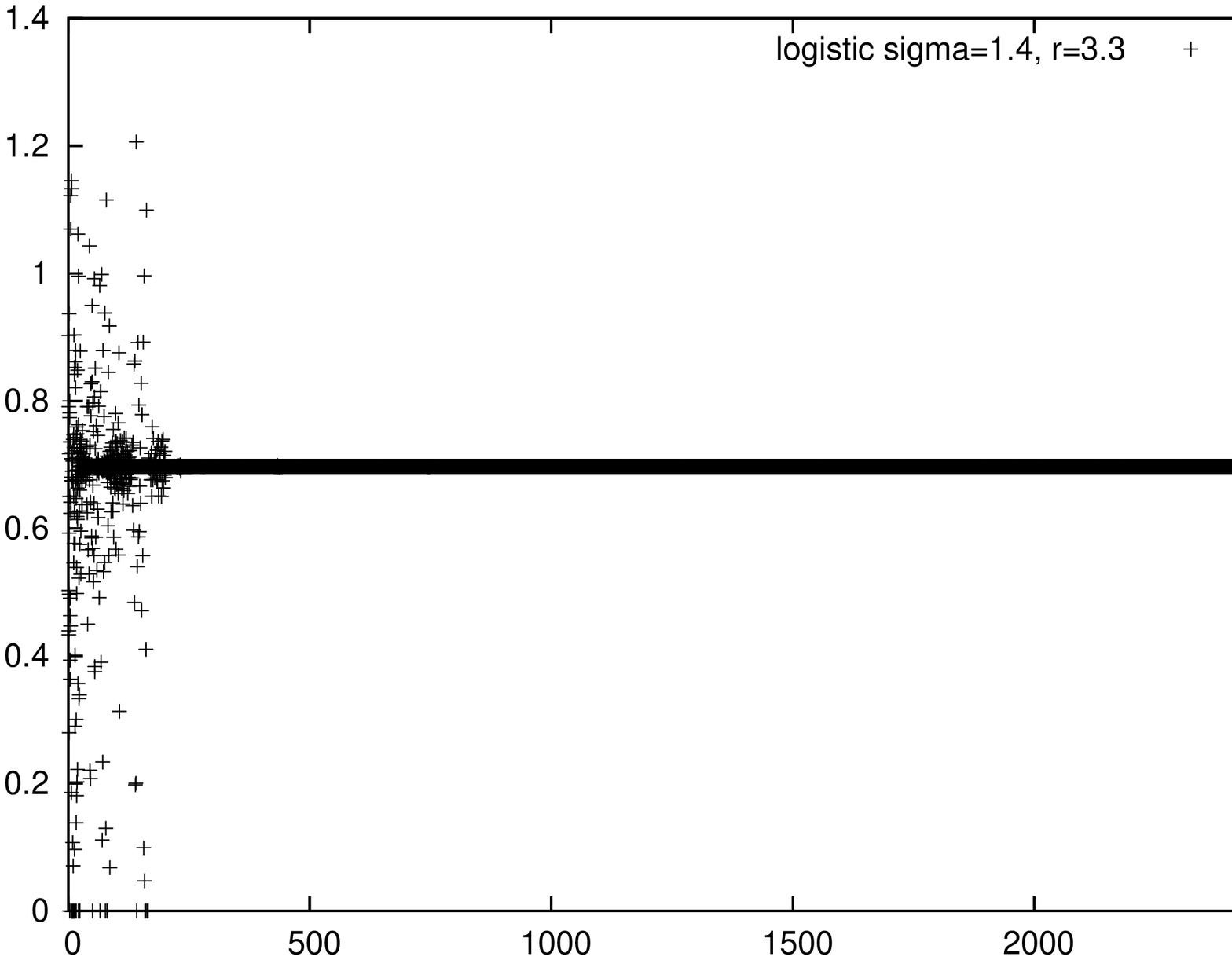}
\vspace{7mm}

\caption{Five runs of the perturbed logistic equation (\protect{\ref{eq:log}}) with $r=3.3$, uniformly distributed in $[-1,1]$ noise, $x_0=0.4$ and $\sigma=0.8,1.2,1.3,1.4$, respectively.
}
\label{figure2}
\end{figure}
\end{example}


Next, we proceed to the logistic map with a Bernoulli distributed $\xi$.

\begin{example}
Let $r=3.3$, then $q=0.3$, $K=0.7$. Substituting $\sigma=0.6$ in \eqref{calc:Evarbernpar} we conclude that 
stabilization is not guaranteed,
see Fig.~\ref{figure6}, left. However, for $\sigma=0.85$ MNC
provides stabilization, as illustrated in Fig.~\ref{figure6}, right. 
\begin{figure}[ht]
\centering
\includegraphics[height=.16\textheight]{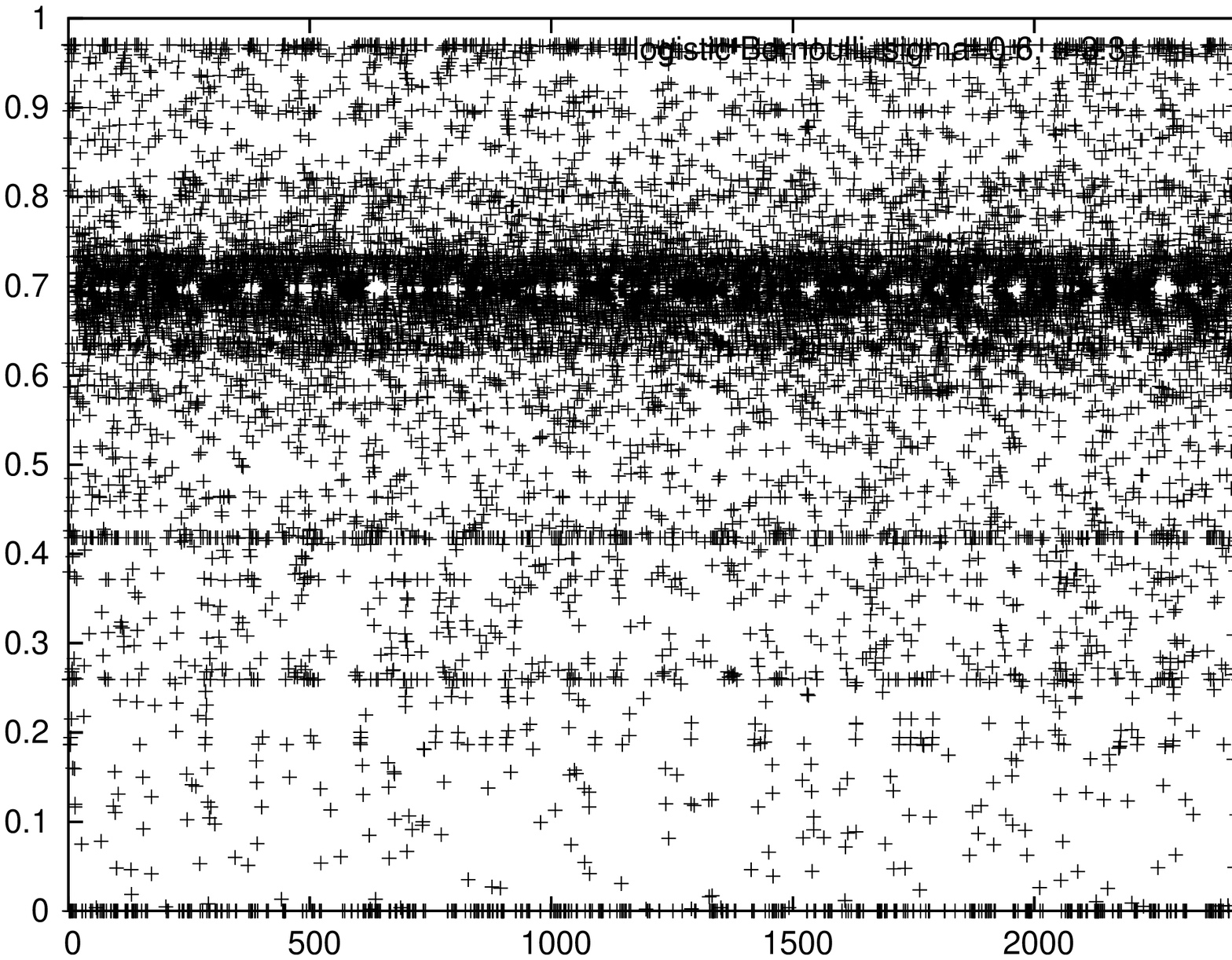}
\hspace{6mm}
\includegraphics[height=.16\textheight]{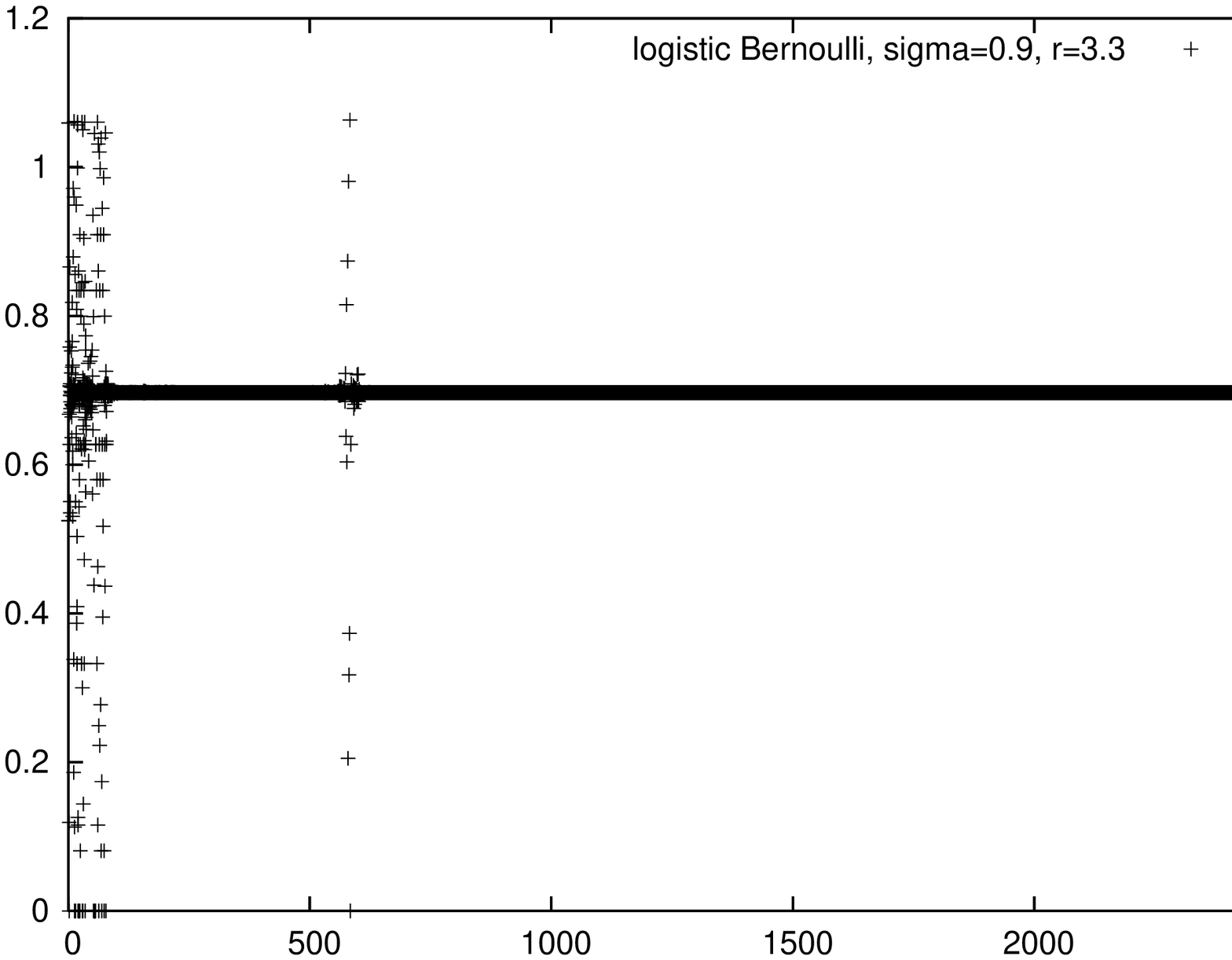}
\vspace{7mm}

\caption{Five runs of the perturbed logistic equation (\protect{\ref{eq:log}}) with noise (\protect{\ref{def:bernxi}}),
$r=3.3$,
$x_0=0.4$ and $\sigma=0.6$ (left) and $\sigma=0.85$ (right).
}
\label{figure6}
\end{figure}
\end{example}

\subsubsection{The logistic map: stabilization with the combined method}

Finally, we consider the case when a combination of DWM and MNC is required to achieve stabilization.

\begin{example}
\label{ex_chaos_1}
For  $r=4$ we have  $q=1$. We can show that the parameters $\sigma=2.1$, $\lambda=0.4457$, $u=0.05$, $\beta=0.00617$ satisfy 
\eqref{calc:Evarbernpar}, \eqref{as:Thetalambda}, \eqref{est:beta}.

In order to estimate $\delta$ by \eqref{def:delta} we need to have $\bar N$ and $M$, 
see \eqref{ineq:prob1}, \eqref{def:barTvM}, \eqref{def:barN2}. 
For $\gamma=0.9$ using  Chebyshev's inequality we get $\bar N=272$, 
by the normal approximation we have $\bar N=74$. 
The number $M=\bar \Theta ^{\bar N-1}=4.1^{\bar N-1}$ is big even for  $\bar N=74$, which implies that $\delta$ is very small.
So in this case it is reasonable to apply the combined  method.

In numerical implementations, we choose $\sigma = 2.1$, 
$u=0.125$, $\beta=0.015$. For the parameters of DWM,
we use $a=1.25$, $\displaystyle \alpha_j= 2.2 a^{-(j+1)}u$, see Fig.~\ref{figure7}.
Similarly to Section~\ref{sec:by_noise}, we illustrate by simulations that stabilization is possible even in the case when
the procedure is more general than theoretically described. In Section~\ref{sec:by_noise}, we 
stabilized a positive equilibrium with noise only. 
Here, we implement a combined algorithm with $u$ and $\beta$ 
larger than theoretically predicted, which is illustrated in Fig.~\ref{figure7}.
\begin{figure}[ht]
\centering
\includegraphics[height=.16\textheight]{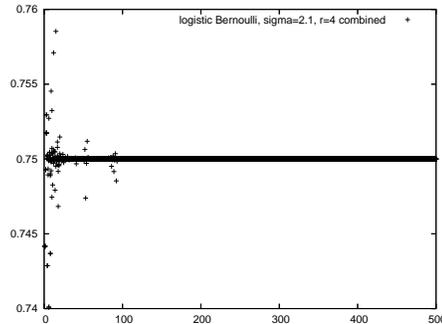}
\vspace{7mm}

\caption{Six runs of perturbed logistic equation (\protect{\ref{eq:log}})
with $r=4$ and $\sigma=2.1$, $u=0.125$, $\beta=0.015$ 
and  noise (\protect{\ref{def:bernxi}}). 
}
\label{figure7}
\end{figure}
\end{example}


\subsection{The Ricker map}

Consider the Ricker equation
\begin{equation}
\label{eq:Ricker}
z_{n+1}={\rm f}(z_n)=z_n e^{r(1-z_n)}, \quad z_n\in [0, \infty), \quad n \in {\mathbb N}_0,
\end{equation}
which has a positive equilibrium at $K=1$, $f'(1)=1-r$, $q=r-2$.

\begin{example}
\label{ex_Rick1}
Let  $r=2.3$, then $q=0.3$. First, consider $\sigma=1.2$ 
with $\xi$ continuous uniformly   distributed in $[-1,1]$.
Applying  \eqref{calc:Eunif}  we get $\lambda=-0.05<0$,
the same inequality is valid for $\sigma=1$.
Fig.~\ref{figure3} illustrates that we do not observe stabilization for $\sigma=1$ (upper left)
and $\sigma=1.2$ (upper right).
For $\sigma=1.22$ and $\sigma=1.24$, we also get similar inequalities:  $\lambda=-0.037<0$ and 
$\lambda=-0.023<0$, respectively.
However, on Fig.~\ref{figure3} for $\sigma=1.22$ (lower left) and $\sigma=1.24$ (lower right) we observe stabilization, which does not contradict  to our theory, as the results are just  sufficient.
\begin{figure}[ht]
\centering
\includegraphics[height=.16\textheight]{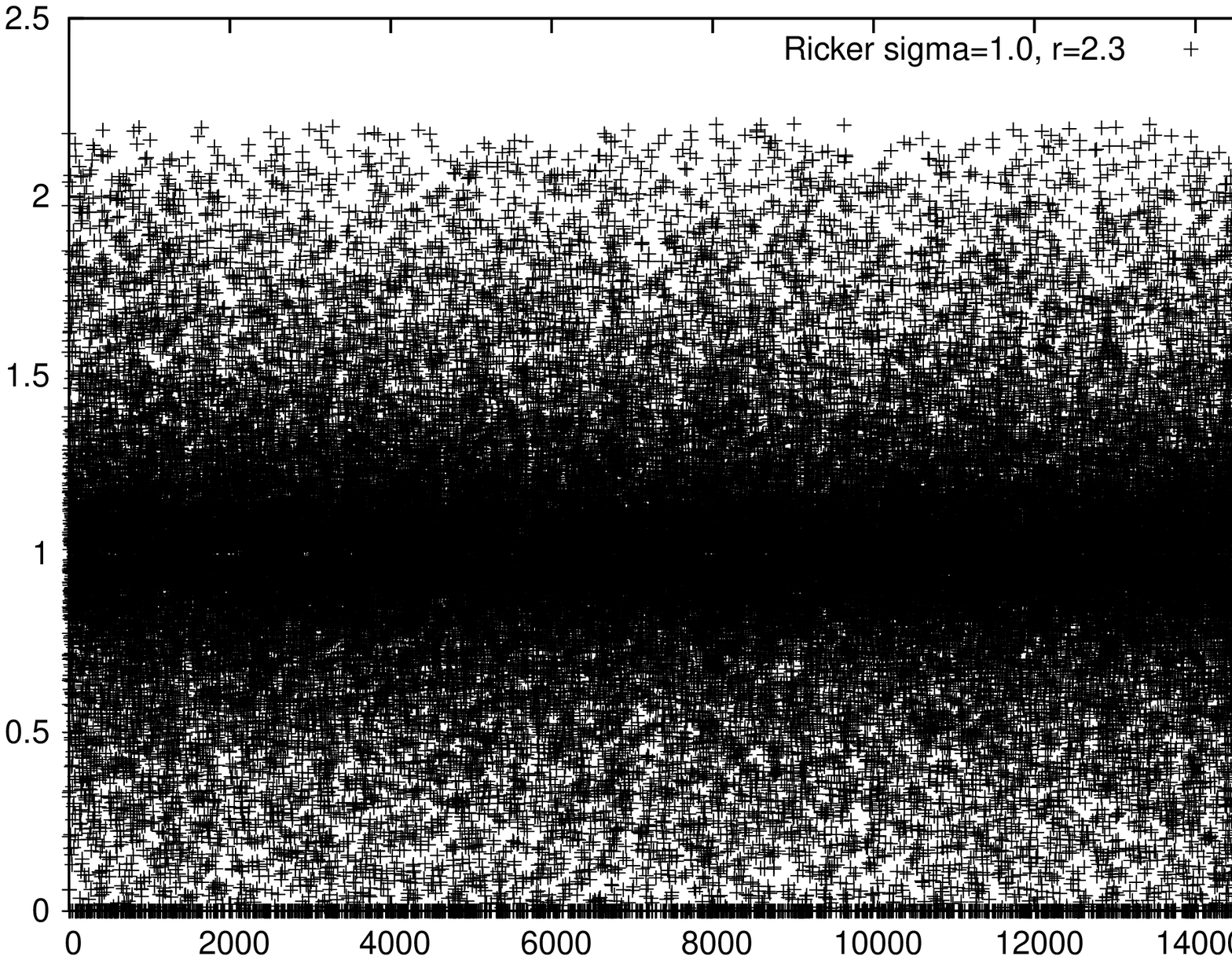}
\hspace{6mm}
\includegraphics[height=.16\textheight]{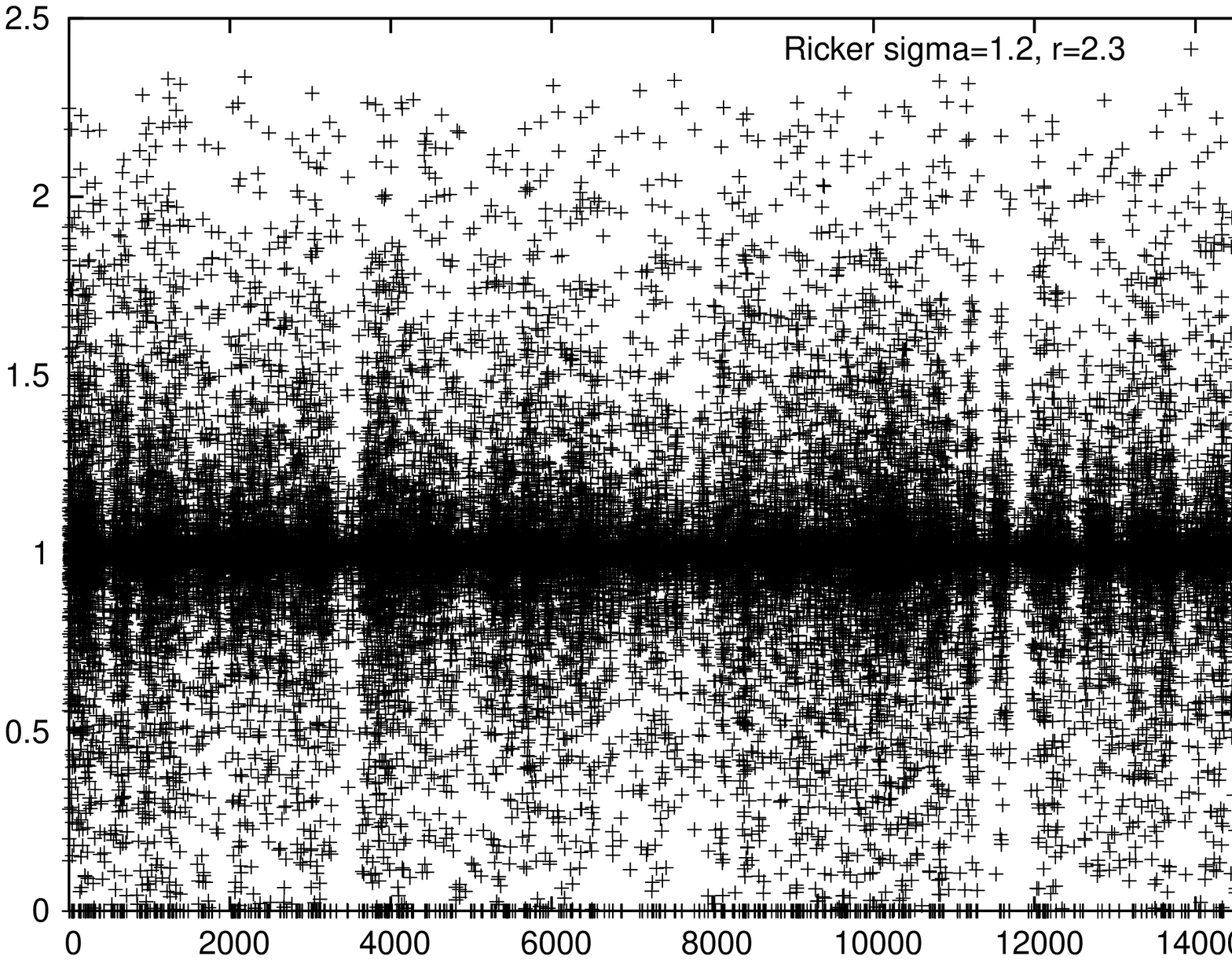}
\vspace{6mm}

\includegraphics[height=.16\textheight]{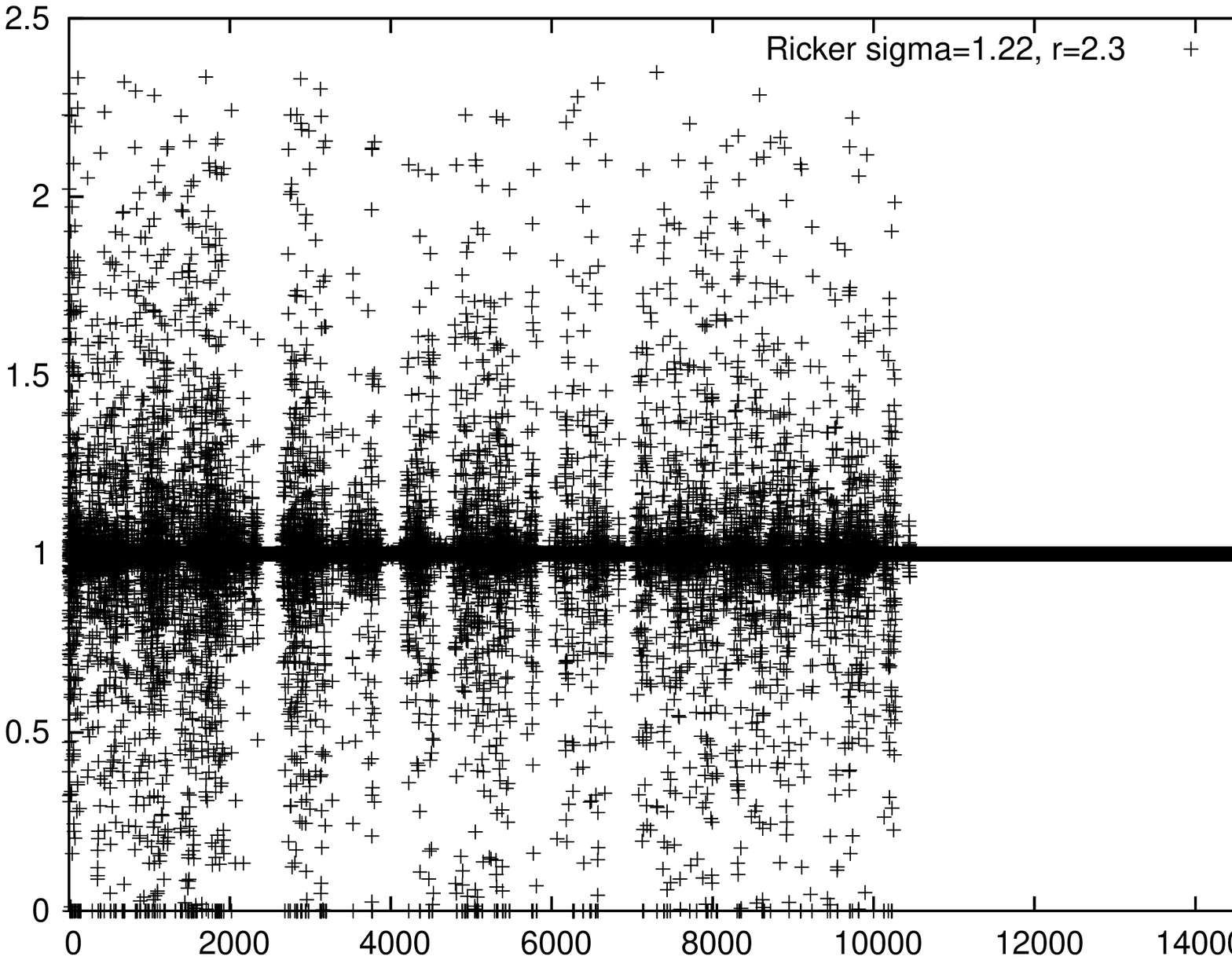}
\hspace{6mm}
\includegraphics[height=.16\textheight]{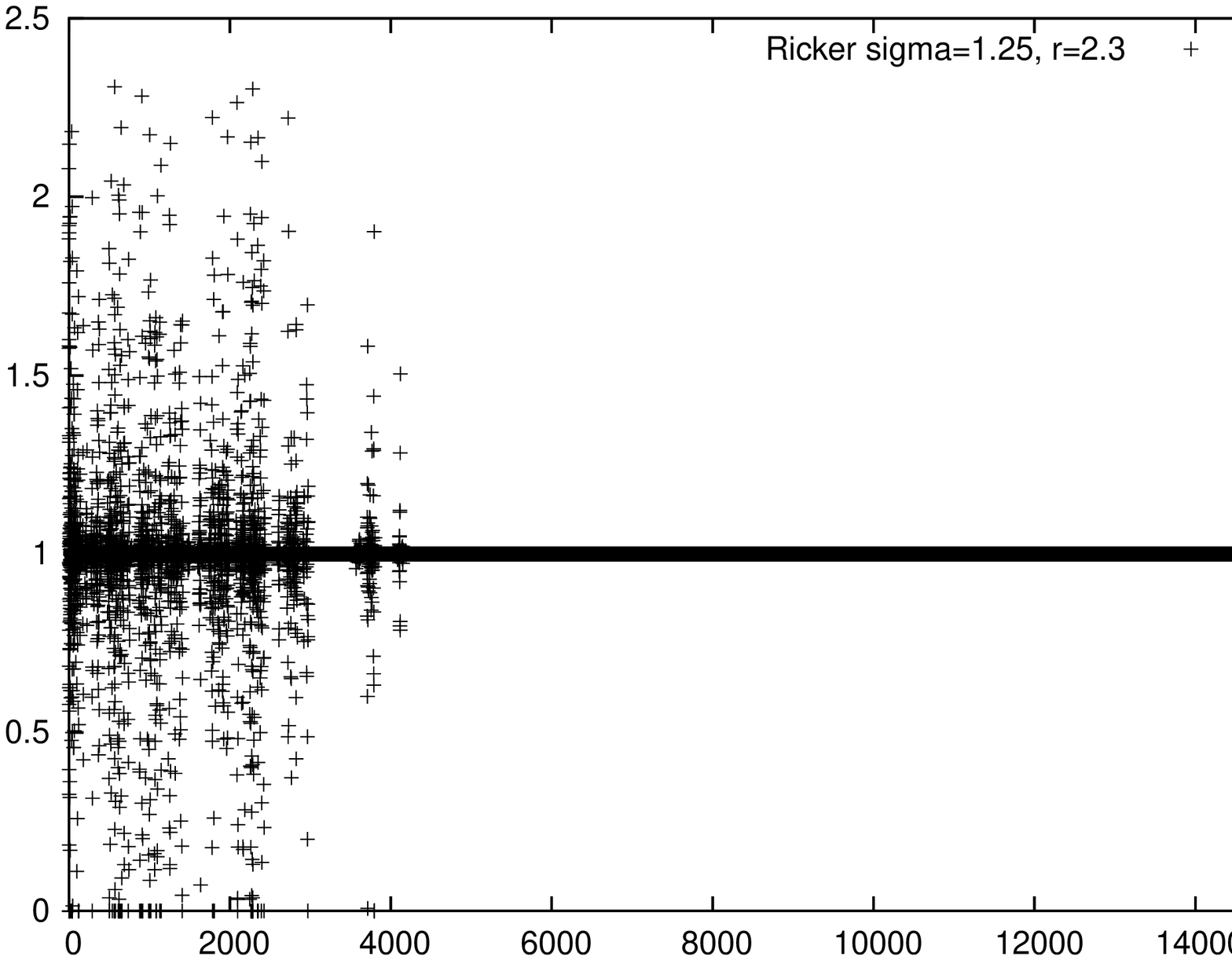}
\vspace{7mm}

\caption{Five runs of the perturbed Ricker equation (\protect{\ref{eq:Ricker}}) with $r=3.3$, uniformly distributed in $[-1,1]$ noise, $x_0=0.4$ and $\sigma=1.0,1.2,1.22,1.24$ (from left to right), respectively.
}
\label{figure3}
\end{figure}
\end{example}

Next, we proceed to a larger $r$, for which the value of $\delta$ is extremely small and just MNC, starting from 
a bigger than $I_\delta$ interval,  cannot stabilize a solution.
Similarly to the logistic case, we consider the Bernoulli distribution.

For $r\ge 3$ and $u$ small enough, we have,
$$
\max_{z\in  I_{u, K}}|{\rm f}'(z)|=|{\rm f}'(1-u)|=e^{r(1-1+u)}|1-r(1-u)|=e^{ru}(r(1-u)-1),
$$
$$
\min_{z\in  I_{u, K}}|{\rm f}'(z)|=|{\rm f'}(1+u)|=e^{r(1-1-u)}|1-r(1+u)|=e^{-ru}(r(1+u)-1).
$$
So we can set
$\displaystyle
\underline L:=e^{-ru}(r(1+u)-1), \quad 1+\bar q:=e^{ru}(r(1-u)-1)
$
and find such $u$ that $\underline L>\bar q.$  We claim that, for $r=3$, it is possible to take $u=0.1$. Indeed, $1+\bar q=2.278$, $\underline L=1.71$ and $\bar q=2.278-1=1.278<1.71=\underline L.$

So on $I_{0.1,\,  1}$  estimate \eqref{cond:Lqno2}  takes place with 
$\bar q=1.278$, $\underline L=1.71,$
and we can find $\beta$ for construction of the DWC algorithm.
However, in numerical runs, we illustrate that a bigger $u$ can work as well.

\begin{example}
\label{ex_chaos_2}
Let $r=3$, $\sigma = 2.1$,  with noise \eqref{def:bernxi}, then condition \eqref{calc:Evarbernpar}  holds. 
We take $u=0.125$ and $\beta=0.015$. For the parameters of DWC, 
we use $a=1.25$, $\displaystyle \alpha_j= 3 a^{-(j+1)}u$. In Fig.~\ref{figure8}, we observe stabilization.
Note that the map is chaotic.

\begin{figure}[ht]
\centering
\includegraphics[height=.18\textheight]{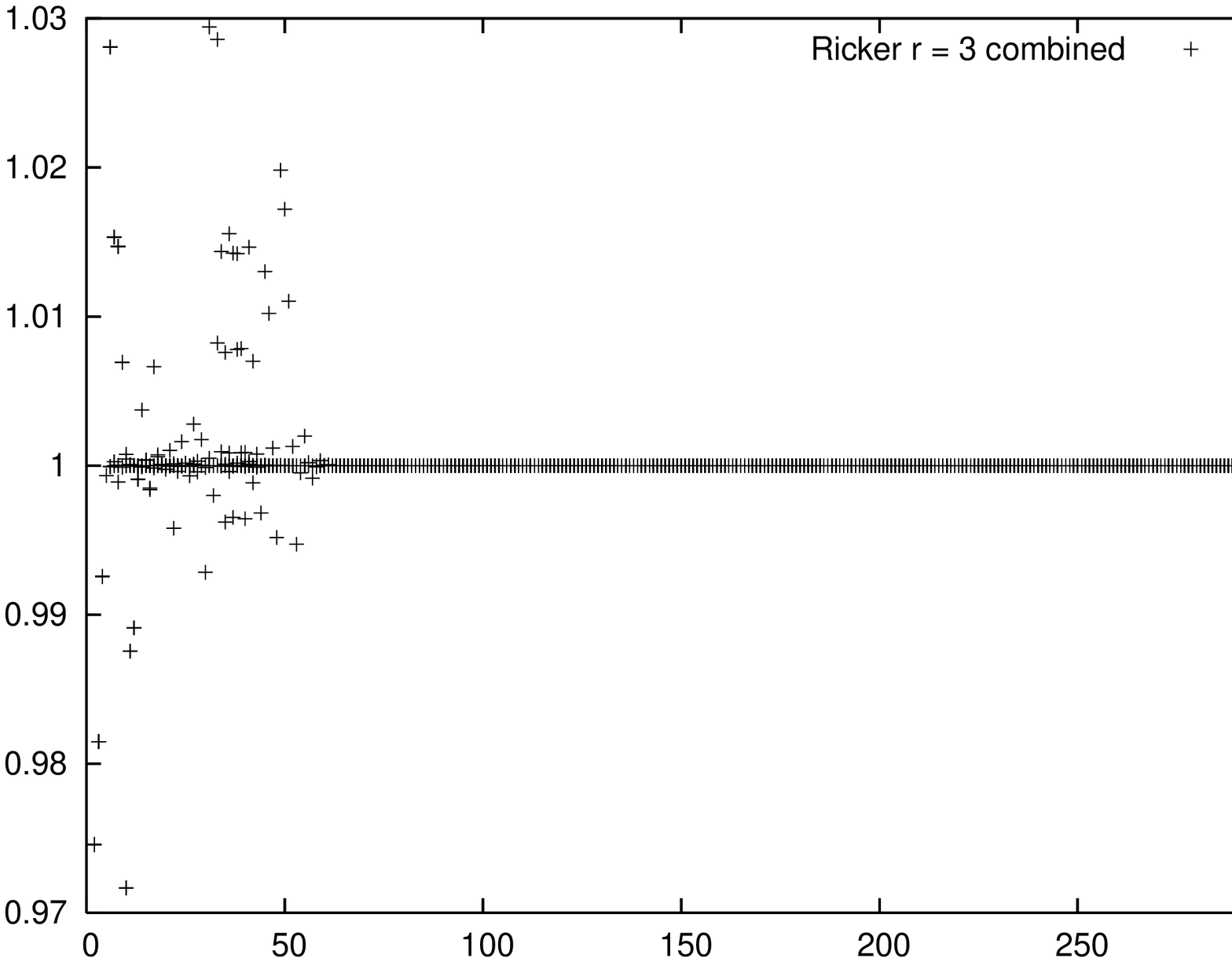}
\vspace{7mm}

\caption{Six runs of perturbed Ricker equation (\protect{\ref{eq:Ricker}})
with $r=3$ and $\sigma=2.1$, $u=0.125$, $\beta=0.015$, noise (\protect{\ref{def:bernxi}}). 
}
\label{figure8}
\end{figure}
\end{example}

\subsection{An example with $\kappa<1$}

Finally, we consider an unbounded function for which, to the best of our knowledge, none of the previous results 
allows to prove the possibility of stabilization by noise.

\begin{example}
Consider the equation
\begin{equation}
\label{ex9eq1}
x_{n+1}=1.5 (x_n-1)+0.5 |x_n-1|^{3/2}+1+\sigma(x_n-1)\xi_{n+1}
\end{equation}
which involves $\rm f$ satisfying \eqref{repr:rmfintr} with $q=0.5$, $K=1$, $\kappa=0.5$. 
Thus, $\rm f$ is not Lipschitz  at $K=1$, and all the results \cite{Medv} fail for this case. 

We use either uniform continuous or Bernoulli distribution in \eqref{ex9eq1} with $\sigma=1.8$, see Fig.~\ref{figure8a}, 
left and right, for the initial value $x_0=1+10^{-11}$ and $x_0=1+10^{-8}$, respectively.

\begin{figure}[ht]   
\begin{center}
\includegraphics[height=0.3\textheight]{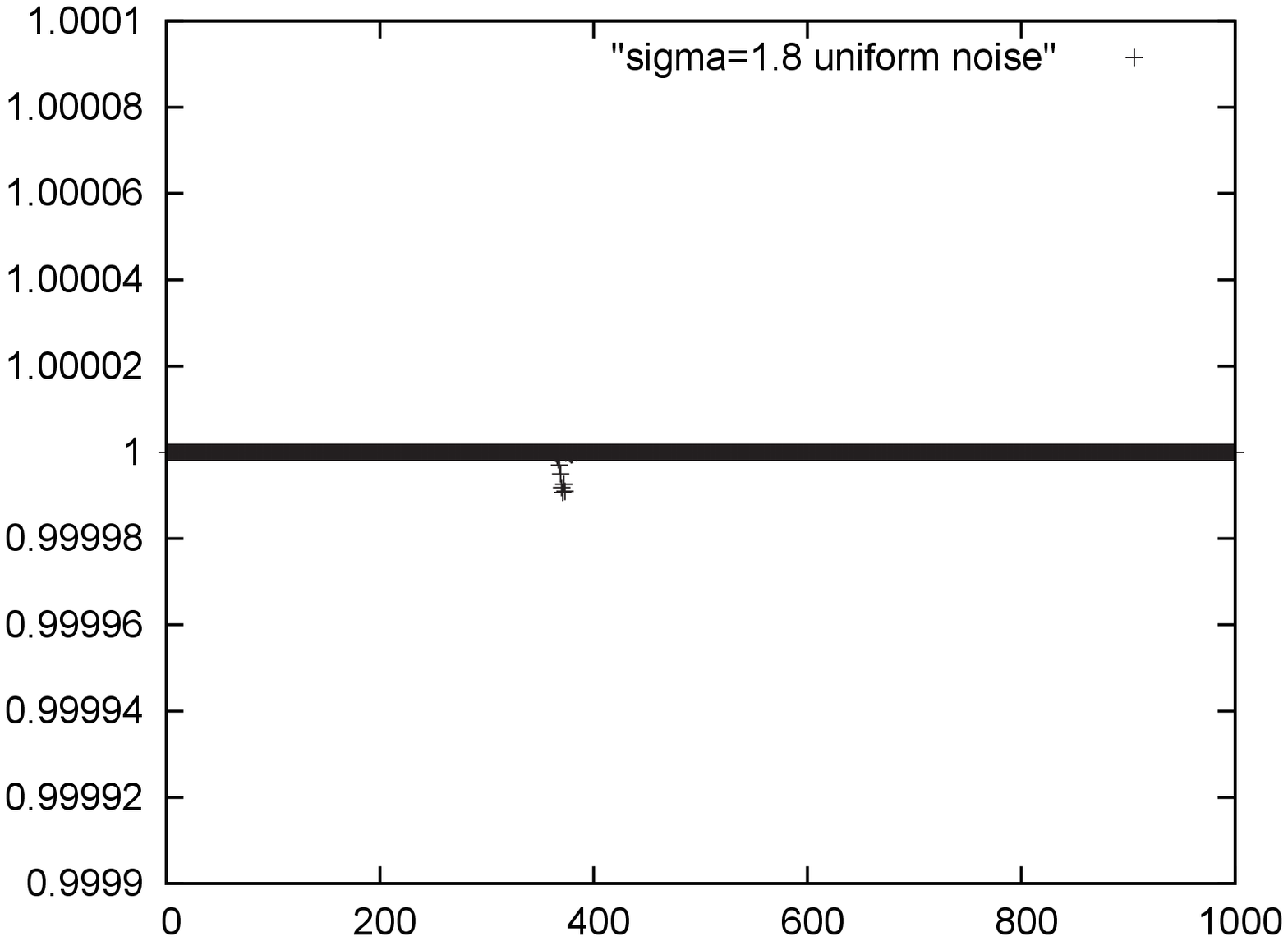}
\includegraphics[height=0.3\textheight]{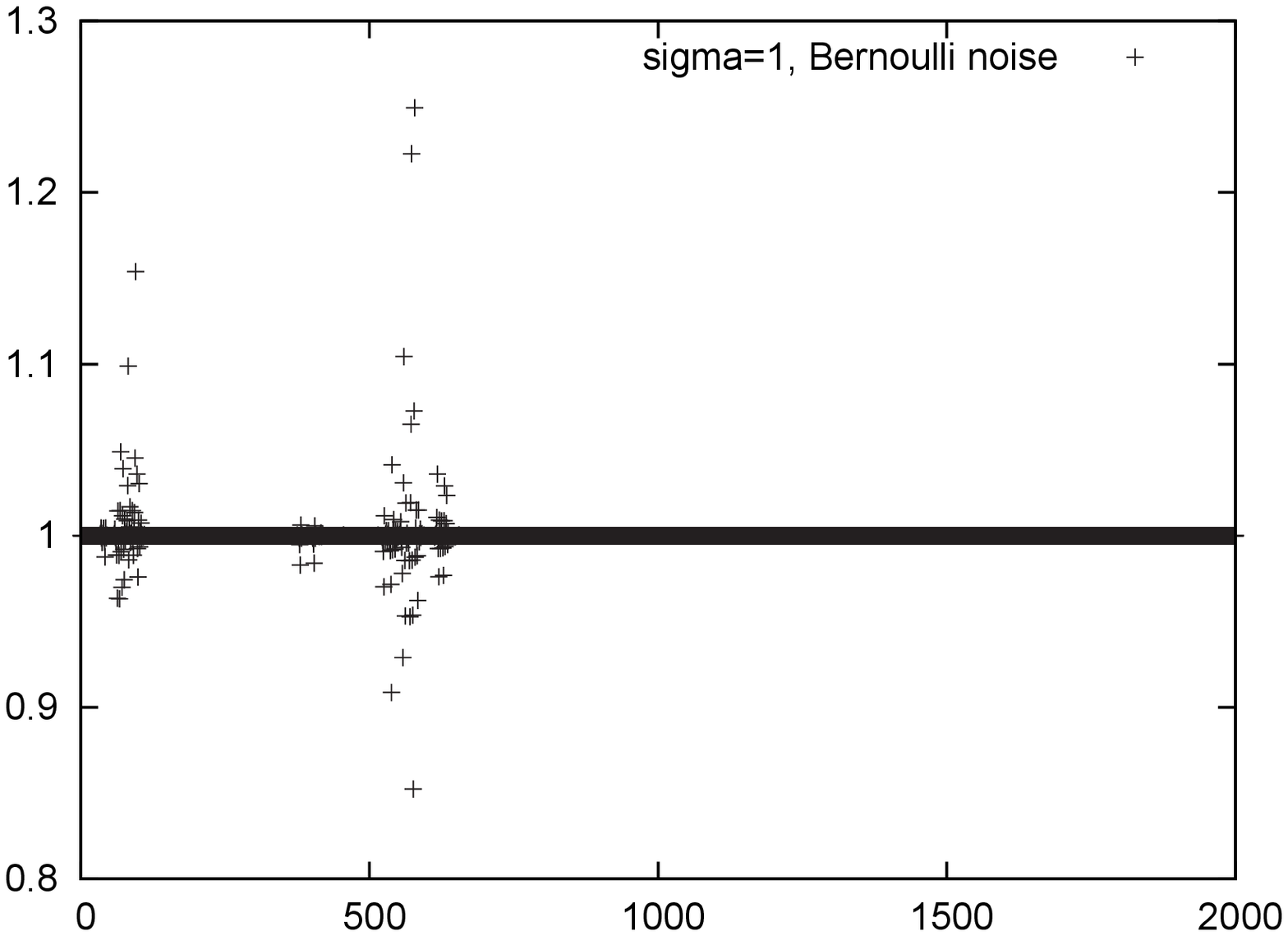}
\end{center}
\caption{Five runs of (\protect{\ref{ex9eq1}}) for the uniform distribution, $x(0)=1+10^{-11}$ (left) noise (\protect{\ref{def:bernxi}}) with
$x(0)=1+10^{-8}$ and $\sigma=1.8$.}
\label{figure8a}
\end{figure}

\end{example}



\section{Discussion}
\label{sec:discussion}

We have proved that, when the equilibrium is unstable which includes the scenarios of
stable cycles with high amplitude or even chaos, 
we still can stabilize equations, for example,  by using the Bernoulli type of noises \eqref{def:bernxi},
if the initial value is in a neighbourhood $I_{\delta, K}$ of the positive equilibrium $K$. In simulations, we also considered a continuous uniformly distributed noise.
However, the stronger the instability, the  bigger the noise coefficient we need to take, 
which implies that the interval $I_{\delta, K}$ of the initial values becomes very small.  
To avoid the necessity to bring a solution to $I_{\delta, K}$, we  apply a combined method and start a solution from a bigger interval 
$I_{\beta, K}$ pushing it back into $I_{\beta, K}$ each time when it gets out to $I_{u, K}$. We have justified that with the initially fixed 
probability, the combined method will deliver the solution into $I_{\delta, K}$ in a finite number of steps. After 
that  moment we apply multiplicative noise control which eventually brings a solution to $K$.

The results of the paper can be summarized as follows.

\begin{enumerate}
\item
We developed a two-step method, where at the final stage stabilization is achieved only by applying a multiplicative noise. 
However, numerical simulations illustrate that, for a wide range of problems, where there is a stable two-cycle in a non-controlled
equation, stabilization by noise is possible, without any additional preliminary methods.
Assumption~\ref{as:lambda} is absolutely crucial for the possibility to stabilize by noise, and the stronger is instability (measured, for example, by the distance of the map parameter from the bifurcation point where stability is lost), a larger noise amplitude should be chosen to guarantee stabilization. 
\item
In some works, e.g. \cite{Medv}, stabilization is considered in the case when stability switches between the two equilibrium 
points, and the value of the bifurcation parameter is quite close to the bifurcation point, in particular, $q=0.05$. 
First, we consider stabilization of an arbitrary point, allowing a period-doubling bifurcation.
Second, stabilization was achieved quite far from the bifurcation point where the considered equilibrium point lost stability: $q=0.1,0.3$ and even $q=1$. In Examples~\ref{ex_chaos_1} and \ref{ex_chaos_2}, the non-controlled equations were chaotic. In order to stabilize, we successfully implemented the combined method.
\end{enumerate}


Note that Assumption~\ref{as:fu} assumes existence of exactly one positive equilibrium (which is satisfied for Ricker and logistic maps). However, the results of the present paper can be applied to the case of several positive equilibrium points, once, using a deterministic control, we bring a solution in a neighbourhood of a chosen equilibrium $K_1$ and set up a stabilizing stochastic perturbation dependent on $x_n-K_1$. Then, for any $x_0 >0$, a solution converges to $K_1$.


Two other most substantial extensions of the present work are listed below.
\begin{itemize}
\item
In the present paper, we stabilized scalar first order difference equations. It would be interesting
to consider stabilization of higher order equations and systems by noise. It is natural to expect that 
a solution should be in a neighbourhood of the equilibrium, in order to achieve stabilization.
Thus, some deterministic or stochastic method bringing a solution into a target domain, should precede
stabilization by noise.
\item
So far we only stabilized an equilibrium. An interesting question is whether it is possible to stabilize an unstable 
cycle. It would probably be even more interesting to prove that, under certain conditions,  a stochastic perturbation reduces a cycle amplitude, and construct relevant estimations.
\end{itemize}

\section{Acknowledgment}
The authors thank the referees for valuable comments which greatly contributed to the presentation of the paper.
Both authors are grateful to  the American Institute of Mathematics SQuaRE program during which this research was initiated.  The first author was supported by the NSERC grant RGPIN-2015-05976.



\section{Appendix A}

\label{sec:expest}

In this section we derive exponential estimates of MNC solutions and prove Theorem \ref{thm:main1}.

We consider the case when the solution has already reached 
the target interval $I_{\delta,K}=(K-\delta,K+\delta)$ which happened at some random moment $\tau$. 
For simplicity of calculations, in this section we assume that the equilibrium $K$ equals zero, 
so instead  of $z_n$ we denote the solution by $x_n$.
 
We assume everywhere below that  $\mathbb P\{\tau<\infty\}=1$, as we prove it for the Directed Walk Control in Appendix B.
Once this happens, we apply MNC (Multiplicative Noise Control) which is based on
the application of Kolmogorov's  Law of Large  Numbers (see Sections \ref{subsec:Kolm} and \ref{subsec:rmffxi}).  
So we consider  equation \eqref{eq:stoch21} with the initial value $x_\tau\in I_\delta$, 
where $\tau$  is the first  moment when $x$  enters  $I_\delta$. The fact that, in general, $\tau$ is a  random variable, brings some extra technical difficulties.  

In Section~\ref{subsec:3gr} we  introduce two classes of  sets, which will be helpful in the proofs. 
The first class consists of sets,  for each of which the value of $\tau$ is constant. 
Construction of the  second class is based on 
Kolmogorov's  Law of Large Numbers  and 
on the number   for which a desired estimate on the sums of
$\ln|\Theta_i|$, where $\Theta$ are defined in \eqref{def:Theta},  holds with a given  probability.
Then we evaluate probabilities of certain intersections and unions of those sets.

In Section~\ref{subsec:defrel} we define several  parameters, 
in particular $\delta$, the radius of the initial interval which guarantees local asymptotic stability. 
In Section~\ref{subsec:Expxn1}, after obtaining a  general estimate for solutions of the equation with MNC, we derive  an exponential estimate for the solution with a certain 
value at a step $\tau$ treated as an initial value. In  Sections~\ref{subsec:tau0}  
we discuss the case   $\mathbb P\{\tau=0\}=1$, which means that, with probability 1, the solution starts from the interval $I_\delta$. And, finally, in  Sections~\ref{subsec:Prmain1}  we give the  proof of Theorem \ref{thm:main1}.

\subsection{Random moment $\tau$ and classes of  probability sets}
\label{subsec:3gr}

%

%

 \bigskip

\subsubsection{Definition of $\Omega_{\tau k}$} 

For a random variable $\tau$ satisfying  Assumption \ref{as:tau1} and for  each  $k\in \mathbb N_0$, we set 
\begin{equation}
\label{def:omegaktau}
\Omega_{\tau k}:=\left\{\omega\in \Omega: \tau=k  \right\}, \quad \text{so} \quad \Omega=\bigcup_{i=0}^{\infty}  \Omega_{\tau i}.
\end{equation}
Since $\Omega_{\tau k}$ are mutually exclusive, we also have $\sum_{i=0}^{\infty}  \mathbb P\Omega_{\tau i}=1.$


\subsubsection{Definition of $\Omega_{\gamma k}$} 
\label{sec:def:Ok} 

Fix $\gamma \in (0, 1)$ and, based on \eqref {ineq:updown}, find a nonrandom number $\bar N=\bar N(\gamma, \varepsilon)$ such that\begin{equation}
\label{ineq:prob1}
\mathbb P \left\{-(\lambda+\varepsilon)n<\sum_{i=1}^{n}  v_i\le -(\lambda-\varepsilon)n, \quad \text{for all} \quad  n\ge \bar N \right\} > 1-\frac \gamma 2.
\end{equation}
For each $k\in \mathbb N_0$, we set
\begin{equation}
\label{def:Ok}
\Omega_{\gamma k}:=\left\{\omega \in \Omega: -(\lambda+\varepsilon)n<\sum_{i=k+1}^{k+n}  v_i\le -(\lambda-\varepsilon)n, \quad \text{for all} \quad  n\ge \bar N  \right\}.
\end{equation}
In general,  $\Omega_ {\gamma k}\neq \Omega_{\gamma j} $ for $k\neq j$, but since $v_n$ are identically distributed,  
we have $\mathbb P \left\{\Omega_{\gamma k}\right\}=\mathbb P \left\{\Omega_{\gamma j}\right\}$, $\forall k, j\in \mathbb N_0$, 
where $\mathbb P \left\{\Omega_{\gamma 0}\right\}$ is estimated in \eqref{ineq:prob1}.
Then, for all $k\in \mathbb N_0$,  
\begin{equation}
\label{ineq:probk}
\mathbb P \left\{\Omega_{\gamma k}\right\}> 1-\frac \gamma 2.
\end{equation}
By \eqref{def:Ok}, for all $n\ge \bar N$  and all $k\in \mathbb N_0$ we have
on  $\Omega_{\gamma k}$,
\begin{equation}
\label{ineq:probk1}
e^{-(\lambda+\varepsilon)n}\le  e^{\sum_{i=k+1}^{k+n}  v_i} \le  e^{-(\lambda-\varepsilon)n}.
 \end{equation}

\medskip

\subsubsection{Estimation of some probabilities} 

Let $\{\Omega_{\tau k}\}_{k\in \mathbb N}$, $\{\Omega_{\gamma k}\}_{k\in \mathbb N}$ be defined by 
\eqref{def:omegaktau} 
and   \eqref{def:Ok}, respectively.  Due to the fact that $\tau$ is independent of $v_i$,  we have
for nonzero left-hand side,
\[
\mathbb P \left\{ \Omega_{\tau k}\cap \Omega_{\gamma k}\right\}=\mathbb P \left\{ \Omega_{\tau k}\right\}
\mathbb P \left\{  \Omega_{\gamma k}\right\} >\left(1-\gamma\right)\mathbb P \left\{ \Omega_{\tau k}\right\}.
\]
Note that,  for $k\neq i$, the  sets $ \Omega_{\tau k}\cap \Omega_{\gamma k}$ and $ \Omega_{\tau i}\cap \Omega_{\gamma i}$ are 
mutually exclusive since $ \Omega_{\tau k} \cap \Omega_{\tau i}= \emptyset$, $i \neq j$. 
Denote, for $s\in \mathbb N$,
\begin{equation}
\label{def:taugamman3}
\Omega_{\tau, \gamma}:=\bigcup_{k=0}^{\infty}\left[\Omega_{\tau k}\cap \Omega_{\gamma k}\right], \quad \Omega_{\tau, \gamma, s}:=\bigcup_{k=0}^{s}\left[\Omega_{\tau k}\cap \Omega_{\gamma k}\right].
\end{equation}
Then, applying \eqref{def:omegaktau}, we get
\begin{equation}
\begin{split}
\label{est:taugamma}
\mathbb P \left\{\Omega_{\tau, \gamma}\right\}&
=\sum_{i=0}^{\infty}\mathbb P \left\{\Omega_{\tau i}\cap \Omega_{\gamma i}\right\}=\sum_{i=0}^{\infty}\mathbb P \left\{\Omega_{\tau i}\right\} \mathbb P \left\{\Omega_{\gamma i}\right\}\\&> (1-\gamma)\sum_{i=0}^{\infty}\mathbb P \left\{ \Omega_{\tau i}\right\}=1-\gamma.
\end{split}
\end{equation}
The proof of the following lemma is standard and thus is omitted.
\begin{lemma}
\label{lem:estg}
Let $\gamma\in (0, 1)$ and $\Omega_{\tau, \gamma, s}$ be defined as in \eqref {def:taugamman3}. 
Then there exists $n_0\in \mathbb N_0$ such that
\begin{equation}
\label{est:n3}
\mathbb P \left\{\Omega_{\tau, \gamma, n_0}\right\}\ge 1-\gamma.
\end{equation}
\end{lemma}


\subsubsection{Definition and estimation of $\mathcal B_{n, k}$ and $\mathcal A_k$ for  $n\le \bar N$} 

For each $n\in \mathbb N$,  $k\in \mathbb N_0$,  based on the linear part of equation \eqref{eq:stoch21},  we construct  non-negative random variables $\mathcal B_{n, k}$ and $\mathcal A_k$  by
\begin{eqnarray}
\label{def:Jn}
&&\mathcal B_{n, k}:=\prod_{i=k+1}^{n+k}|\Theta_i|=e^{\sum_{i=k+1}^{n+k}v_i},\\\label{def:calA1}
&&\mathcal A_k:= \max_{i, n: \, 1\le i\le n\le \bar N}\prod_{j=i+k}^{n+k}|\Theta_j|=\max_{i, n: \, 1\le i\le n\le \bar N} e^{\sum_{i=k+i}^{n+k}v_i},
\end{eqnarray}
where $\bar N>1$  is such that \eqref{ineq:prob1}  and \eqref{ineq:probk} hold.  By Assumption \ref{as:chixi}  and 
\eqref{rel:boundThetav},   random variables $\Theta_n$ and $v_n$ are bounded by $\bar \Theta$ and $\bar v$, respectively. If we set
\begin{equation}
\label{def:barTvM}
M:=e^{(\bar N-1)\bar v}=\bar \Theta^{\bar N-1},
\end{equation}
then $M>1$, and for all $\omega \in \Omega$, $n\in \mathbb N$,  $k\in \mathbb N_0$, 
\begin{equation}
\label{est:Ank}
\mathcal B_{n, k}=e^{\sum_{j=k+1}^{n+k} v_j}\le e^{(n-1)\bar v}, \quad  \mathcal A_k
\le \max_{i, n: \, 1\le i\le n\le \bar N} e^{(n-i)\bar v}  \leq e^{(\bar N-1)\bar v}=M.
\end{equation}

Note  that for  each $k\in \mathbb N_0$  by \eqref{def:Ok} and \eqref{ineq:probk1}, we have for $n\ge \bar N>1$ on $\Omega_{\gamma k}$,
\begin{equation}
\label{est:Bnk}
\mathcal B_{n, k}=e^{\sum_{i=k+1}^{n+k}v_i}\le e^{-(\lambda-\varepsilon)n}<1<M.
\end{equation}

\medskip

\subsection{More definitions and relations}
\label{subsec:defrel}

Let $\lambda$ be defined as in \eqref{as:Thetalambda} and let $C, \kappa$ be from Assumption \ref{as:fu}. Choose some $\varsigma>0$ and $\varepsilon>0$ such that
\begin{equation}
\label{def:varsigmaepsilon}
\varsigma<\frac \lambda 2, \quad \varepsilon < \min\left\{\frac{\kappa \varsigma}2, \, \lambda - \varsigma\right\}.
\end{equation}
Recall that \eqref{ineq:prob1} and \eqref{ineq:probk} hold for $\bar N$. 
Then \eqref{ineq:probk1} implies that for  $n\ge \bar N$, $\omega \in \Omega_{\gamma k}$ and each $k\in \mathbb N_0$, we have
for $\mathcal B_{n, k}$ defined as in \eqref{def:Jn},
\begin{equation*}
e^{-(\lambda +\varepsilon)n} \le\mathcal B_{n, k}\le e^{-(\lambda -\varepsilon)n}, \quad \text{so} \quad \left[\mathcal B_{n, k}\right] ^{-1}\le e^{(\lambda +\varepsilon)n}.
\end{equation*}
By \eqref{def:varsigmaepsilon}, we get $-\lambda+\varsigma +\varepsilon<0$, so
\begin{equation}
\label{est:In1}
e^{(-\lambda+\varsigma +\varepsilon)n} <1.
\end{equation}
For  $n\ge \bar N$  we have on $\Omega_k$, for  each $k\in \mathbb N_0$,
\begin{equation}
\label{est:Ink}
\begin{split}
&\mathcal B_{n, k} \le e^{(-\lambda+\varsigma +\varepsilon)n} e^{-\varsigma n}, \\  &\left[\mathcal B_{n, k}\right] ^{-1} \le e^{(\lambda +\varepsilon)n}=e^{(\lambda -\varepsilon -\varsigma)n}e^{(2\varepsilon+\varsigma)n}=e^{-(-\lambda+\varsigma +\varepsilon)n} e^{(2\varepsilon+\varsigma)n}.
\end{split}
\end{equation}
Since $\kappa \varsigma - 2 \varepsilon>0$ and $-\lambda+\varsigma +\varepsilon <0$,
we can apply  Lemma \ref{lem:aux} with $b_j:=e^{-(\kappa \varsigma - 2 \varepsilon)j}$, 
$e_j:=~e^{-(-\lambda+\varsigma +\varepsilon)(j+1)}$, and  choose a nonrandom $\bar N_1\in \mathbb N$ such that, for 
$n\ge \bar N_1$,
\begin{equation}
\label{def:N2}
e^{(-\lambda+\varsigma +\varepsilon)(n+1)}\sum_{j=1}^n e^{-(\kappa \varsigma - 2 \varepsilon)j}
e^{-(-\lambda+\varsigma +\varepsilon)(j+1)}<\frac 12 C^{-1}e^{-(2\varepsilon+\varsigma)},
\end{equation}
where $C$ is from \eqref{eq:phi} in Assumption~\ref{as:fu}.  Let 
\begin{equation}
\label{def:barN2}
\bar N_2=\max\{\bar N_1, \bar N\},
\end{equation}
$u$ and $\kappa$ be also from Assumption \ref{as:fu}, and $M$ be defined as in \eqref{def:barTvM}.  Choose
\begin{equation}
\label{def:eta}
\eta\le \min\left\{\left( 2CMe^{\varsigma (\bar N_2+1)}\sum_{j=0}^{\bar N_2}e^{-(1+\kappa)\varsigma j}  \right)^{-1/\kappa}, \,\, 1, \,\, u\right\}
\end{equation}
and  
\begin{equation}
\label{def:delta}
\delta \le  \min \left\{ ( M+C)^{-1}e^{-\varsigma } \eta, \,\,\frac 12\eta M^{-1} e^{-\varsigma \bar N_2}\right\},
\end{equation}
which obviously satisfies $\delta<1$ and $\delta<\frac \eta 2$.


\subsection{Exponential estimates of solutions.}
\label{subsec:Expxn1}

In this section we derive exponential estimates of solutions to equations with different initial values. 

\subsubsection{General estimation of solutions}
\label{subsec:genest}

Let $\tau$ satisfy Assumption \ref{as:tau1}, i.e.  $\tau$ be  a positive   integer-valued a.s. finite random variable which is independent of all $\xi_i$, and  let $\nu$ be some  a.s. finite positive  random variable.
Fix some  $k\in \mathbb N_0$ and consider the following modification of equation  \eqref{eq:stoch21} 
\begin{equation}
\label{eq:stochmodk}
y_{n+1}=-(1+q-\sigma \xi_{k+n+1}) y_n+\phi(y_n), \quad y_0=\nu.
\end{equation}

Everywhere further we assume that $\displaystyle \prod_i^j=1$ if $i>j$.
The following lemma is obtained by induction.

\begin{lemma}
\label{lem:genest}
Let $k\in \mathbb N_0$, condition \eqref{eq:phi} hold, $\Theta_i$ be defined as in \eqref{def:Theta}, and $\Omega^*\subseteq \Omega$.   
Assume that $y_n$ is a solution to \eqref{eq:stochmodk} with a random initial value $\nu$ 
such that $\nu\in I_u$ on $\Omega^*$, and, for some $n\in \mathbb N_0$, on $\Omega^*$,
$ y_i\in I_u$, $i=1, 2, \dots, n$.
Then, on $\Omega^*$,
 \begin{equation}
\label{est:gen}
|y_{n+1}|\le e^{\sum_{i=k+1}^{k+n+1}\ln |\Theta_i|}\left(|\nu|+C\sum_{j=0}^n|y_j|^{1+\kappa}e^{-\sum_{i=k+1}^{k+j+1}\ln |\Theta_i|}   \right).
\end{equation}
\end{lemma}

\subsubsection{Exponential estimates of solutions of auxiliary equations}
\label{subsub:aux}

Next, we proceed to exponential estimations for solutions of \eqref{eq:stochmodk}.

\begin{lemma}
\label{lem:expxn}
Let   $\gamma \in (0, 1)$ and $k\in \mathbb N_0$.  Let Assumptions \ref{as:chixi}, \ref{as:fu}, \ref{as:lambda}  and \ref{as:tau1} hold.
Let  the sets $\Omega_{\tau k}$,  $\Omega_{\gamma k}$
and  numbers   $\varsigma$, $\eta$, $\delta$,  be defined by  \eqref{def:omegaktau},  \eqref{def:Ok}, 
\eqref{def:varsigmaepsilon}, \eqref {def:eta} and   \eqref{def:delta}, respectively. Let $y_n$ be a solution  to 
\eqref{eq:stochmodk} with the initial value $y_0=\nu$, where  $\nu$ is a random variable such that  $\nu \in 
I_\delta$   on $\Omega_{\tau k}$. Then, on $\Omega_{\gamma k}\cap \Omega_{\tau k}$, \begin{equation}
\label{est:mainzn}
|y_n|\le \eta e^{-\varsigma n}, \quad  \mbox{for all} \,\,   n\in \mathbb N.
\end{equation}
\end{lemma}

\begin{proof}
Since the sequence $(\xi_n)_{n\in \mathbb N}$ satisfies  Assumptions \ref{as:chixi}, the same holds for  sequences $(\xi_{k+n})_{n\in \mathbb N}$ with any  $k\in \mathbb N_0$. 

We prove \eqref{est:mainzn} by induction.  Note that \eqref{rel:boundThetav} and \eqref{def:barTvM} imply that, for each $k\in \mathbb N_0$,
$
|\Theta_k|\le \bar \Theta \le M.
$
Since $|y_0|\le \delta<u<1$
on $\Omega_{\gamma k}\cap \Omega_{\tau k}$, we can apply \eqref{eq:phi}, recall \eqref{def:Theta} and get  for $n=1$, on $\Omega_{\gamma k}\cap \Omega_{\tau k}$,
\begin{eqnarray*}
|y_1| 
\le e^{\ln |\Theta_{k+1}|}|y_0|+C|y_0|^{1+\kappa}\le  |\nu|( M+C|\nu|^{\kappa})e^{\varsigma }  e^{-\varsigma 
} 
\le  \delta( M+C)e^{\varsigma }  e^{-\varsigma }\le \eta e^{-\varsigma },
\end{eqnarray*}
since $\delta\le ( M+C)^{-1}e^{-\varsigma } \eta$.

Now assume that,  for some $n\in \mathbb N$,  
\eqref{est:mainzn} holds on $\Omega_{\gamma k}\cap \Omega_{\tau k}$ for all $i\le n$,  
and prove that it holds on $ \Omega_{\gamma k}\cap \Omega_{\tau k}$ for $n+1$. 
Then  we can apply estimate  \eqref{est:gen}  from Lemma \ref{lem:genest} with $\Omega^*=\Omega_{\gamma k}\cap \Omega_{\tau k}$,
recall the definition of random variable  $\mathcal B_{n, k}$ in \eqref{def:Jn}
and conclude that, on $\Omega _{\gamma k}\cap \Omega_{\tau k} $, 
\begin{equation*}
\begin{split}
|y_{n+1}|&\le e^{\sum_{i=k+1}^{k+n+1}\ln |\Theta_i|}|\nu|+Ce^{\sum_{i=k+1}^{k+n+1}\ln |\Theta_i|}\sum_{j=0}^{n}|y_j|^{1+\kappa}e^{-\sum_{i=k+1}^{k+j+1}|\Theta_i|} \\
&=\mathcal B_{n+1, k} |\nu|+C\mathcal B_{n+1, k} \sum_{j=0}^{n}|y_j|^{1+\kappa}\left[\mathcal B_{j+1, k} \right]^{-1}:=T^{[k]}_1(n+1)+T^{[k]}_2(n+1).
\end{split}
\end{equation*}
Next, we prove that $T^{[k]}_i(n+1)\le \frac 12\eta e^{-\varsigma n}$ on $ \Omega_{\gamma k}\cap \Omega_{\tau k}$,  for each  $i=1, 2$.
For $\bar N_2$ defined as in \eqref{def:barN2} we  consider two cases:  $n < \bar N_2$ and  $n\ge \bar N_2$.
 
 
{\it Estimation of $T^{[k]}_1(n+1)$.} 
For $n< \bar N_2$, since 
$|y_0|\le \delta<\frac 12\eta M^{-1} e^{-\varsigma \bar N_2}$ and by \eqref{def:calA1}, \eqref{est:Ank},  we have  on $ \Omega_{\gamma k}\cap \Omega_{\tau k}$
\[
T^{[k]}_1(n+1)= \mathcal B_{n+1, k}|\nu|\le M \frac 12\eta M^{-1} e^{-\varsigma \bar N_2} <\frac 12 \eta e^{-\varsigma (n+1)}.
\]
For $n\ge  \bar N_2$, since $\delta<\frac 12\eta $  and by \eqref{est:Bnk}, \eqref{est:In1}, \eqref{est:Ink},  
\[
T^{[k]}_1(n+1)=\mathcal B_{n+1, k}|\nu|\le e^{(-\lambda+\varsigma +\varepsilon)(n+1)} 
e^{-\varsigma (n+1)} |\nu|<e^{-\varsigma (n+1)}\delta<\frac 12\eta e^{-\varsigma (n+1)}.
\]

 {\it Estimation of $T^{[k]}_2(n+1)$.}
For $n< \bar N_2$, by \eqref{def:Jn}, \eqref{def:calA1}, \eqref{est:Ank},  we have on $ \Omega_{\gamma k}\cap \Omega_{\tau k}$,
\begin{equation}
\label{est:T21}
\begin{split}
T^{[k]}_2(n+1)&= C\sum_{j=0}^{n}|y_j|^{1+\kappa} \mathcal B_{n+1, k}\left[\mathcal B_{j+1, k}\right]^{-1}\\&
\le C \eta^{1+\kappa}\sum_{j=0}^{n}e^{-(1+\kappa)\varsigma j}\prod_{i=j+k}^{n+k}|\Theta_i|\le CM \eta^{1+\kappa} \sum_{j=0}^{\bar N_2}e^{-(1+\kappa)\varsigma j}\\&
\le \frac 12 \eta e^{-\varsigma (n+1)}  \eta^{\kappa} 2 CM e^{\varsigma (\bar N_2+1)}  \sum_{j=0}^{\bar N_2}e^{-(1+\kappa)\varsigma j}.
\end{split}
\end{equation}
By substituting estimate \eqref{def:eta} of $\eta$ into \eqref{est:T21} we arrive at  
$$
\begin{array}{ll}
T^{[k]}_2(n+1) \le & \displaystyle  \frac 12 \eta e^{-\varsigma (n+1)} 
\left( 2CMe^{\varsigma (\bar N_2+1)}\sum_{j=0}^{\bar N_2}e^{-(1+\kappa)\varsigma j}  \right)^{-1} \times
\\ & \displaystyle  \times 2 CM e^{\varsigma (\bar N_2+1)}  \sum_{j=0}^{\bar N_2}e^{-(1+\kappa)\varsigma j}
 \le \frac 12 \eta e^{-\varsigma (n+1)}.
\end{array}
$$
For  $n\ge \bar N_2$, applying \eqref{est:Ink} and  \eqref{def:N2}, we get 
\begin{equation*}
\begin{split}
&T^{[k]}_2(n+1)=C \mathcal B_{n+1, k}\sum_{j=0}^{n}|y_j|^{1+\kappa}\left[ \mathcal B_{j+1, k}\right]^{-1}\\&\le C e^{(-\lambda+\varsigma +\varepsilon)(n+1)}e^{-\varsigma (n+1)} \eta^{1+\kappa}\sum_{j=0}^{n}e^{-(1+\kappa)\varsigma j}e^{(2\varepsilon+\varsigma)(j+1)}e^{(\lambda-\varsigma -\varepsilon)(j+1)}\\&
=\frac 12\eta  e^{-\varsigma(n+1)}2C\eta^{\kappa}e^{(-\lambda+\varsigma +\varepsilon)(n+1)}e^{(2\varepsilon+\varsigma)}\sum_{j=0}^{n}e^{-(\kappa \varsigma-2\varepsilon)j}e^{(\lambda-\varsigma -\varepsilon)(j+1)}\\&< \frac 12\eta  e^{-\varsigma (n+1)},
\end{split}
\end{equation*}
which completes the proof.
\end{proof}

\subsection{Some remarks and the case $\tau\equiv 0$.}
\label{subsec:tau0}

\begin{remark}
\label{rem:inu}
By construction of estimate \eqref{est:mainzn} and 
the choice of $\eta$, on $ \Omega_{\gamma k}\cap \Omega_{\tau k}$,  
a solution $y_n$ to equation \eqref{eq:stochmodk}, which starts in $I_\delta$ 
will stay in $I_u$ for all $n\in {\mathbb N}$. 
This implies that, on $ \Omega_{\gamma k}\cap \Omega_{\tau k}$, $y_n$ is also  a solution to the equation
\begin{equation}
\label{eq:stoch2}
y_{n+1}= f(y_n)+\sigma \xi_{k+n+1}y_n, \quad y_0=\nu.
\end{equation}
Note  that $y_n$ does not necessarily belong to  $I_\delta$ for all $n$. However, after at most  $\bar n$ steps, $y_n$ returns to $I_\delta$ and remains there. The number $\bar n$ can be calculated as the minimum $n$  for which $\eta e^{-\varsigma n}<\delta$. In other words,  
\[
\bar n:=\left[ \varsigma^{-1} \ln \left( \frac {\eta}{\delta} \right)  \right]+1.
\]
\end{remark}
If the initial value $\nu$ belongs to the interval $I_\delta$ with probability 1, i.e. $\mathbb P\left\{\nu\in I_\delta\right\}=1$, we have   $\mathbb P\left\{\tau=0\right\}=\mathbb P\left\{\Omega_{\tau 0}\right\}=1$ and  $\mathbb P\left\{\Omega_{\tau k}\right\}=0$ for all $k>0$. So 
\[
\mathbb P\left\{\Omega_{\gamma k}\cap \Omega_{\tau k}    \right\}=0, \quad\forall k>0,
\]
and then, by \eqref{def:taugamman3} and \eqref{est:taugamma}, we have 
\[
1-\gamma< \mathbb P\left\{\Omega_{\tau, \gamma}\right\}=\mathbb P\left\{\bigcup_{k=0}^{\infty}  \left[\Omega_{\gamma k}\cap \Omega_{\tau k} \right] \right\}=\sum_{k=0}^{\infty} \mathbb  P\left\{ \Omega_{\gamma k}\cap 
\Omega_{\tau k}         \right\}= \mathbb P \left\{\Omega_{\gamma 0}\cap \Omega_{\tau 0}    \right\}.
 \]
Thus we arrive at the corollary which can be treated as the main result for $\nu \in I_{\delta}$.

\begin{corollary}
\label{cor:k0}
Let   $\gamma \in (0, 1)$, Assumptions \ref{as:chixi}, \ref{as:fu} and \ref{as:lambda} 
hold, $\varsigma$, $\eta$, $\delta$, be defined by  \eqref{def:varsigmaepsilon}, \eqref {def:eta},  \eqref{def:delta}, respectively, and $x_n$ be a solution  to \eqref{eq:stoch21} with 
$x_0=\nu$, where  $\mathbb P\left\{\nu\in I_\delta\right\}=1$.  Then, with a probability greater than $1-\gamma$,
\begin{equation}
\label{est:mainzn1}
|x_n|\le \eta e^{-\varsigma n}, \quad  \forall n\in \mathbb N.
\end{equation}
\end{corollary}

\begin{remark}
\label{rem:a}
If  $\mathbb P\left\{\nu \in I_\delta\right\}=1$,  estimate \eqref{est:n3} holds with $n_0=0$.
\end{remark}


\subsection{Proof of Theorem \ref{thm:main1}}
\label{subsec:Prmain1}
Set,  as in  \eqref{def:f}, $x_n=z_n-K$. For an arbitrary $k\in \mathbb N_0$, we have 
$\tau(\omega)=k$ on  $\Omega_{\tau k}$, by definition, therefore,   $x_{\tau}(\omega)=x_k(\omega)$  whenever  $\omega\in\Omega_{\tau k}$.  

Denoting $y_n^{[k]}:=x_{k+n}$, we consider on $\Omega_{\tau k}$ the following problem
\begin{equation*}
y^{[k]}_{n+1}=f\left(y^{[k]}_n\right)+\sigma y^{[k]}_n\xi_{k+n+1}, \quad y_0=x_{k}=z_k-K,  \,\, n\in \mathbb N_0,
\end{equation*}
see \eqref{def:f}.
By Remark \ref{rem:inu} (and Lemma \ref{lem:expxn}) we conclude that, 
on $\Omega_{\tau k}$, $y^{[k]}_n$ is also a solution  to \eqref{eq:stochmodk} with $\nu=x_{k}$. 
Lemma \ref{lem:expxn}  and Remark~\ref{rem:inu}  imply the estimates
\begin{equation}
\label{est:zk}
|y^{[k]}_n|\le \eta e^{-\zeta n}, \quad \text{or} \quad |x_{n+k}|\le \eta e^{-\zeta n},  \quad \text{or} \quad 
|z_{n+k}-K|\le \eta e^{-\zeta n},
\end{equation}
which hold on  $\Omega_{\tau k}\cap \Omega_{\gamma k}$,   for each 
$n\in \mathbb N_0$.
Set,  for each $n\in \mathbb N_0$,
\[
z_{n+\tau}(\omega)=z_{n+k}(\omega), \quad \text{when} \quad \omega\in \Omega_{\tau k}\cap \Omega_{\gamma k}, \,\, k\in \mathbb N_0,
\]
then \eqref{est:zk} yields that, for each $n\in \mathbb N_0$,
\begin{equation}
\label{est:ztau}
|z_{n+\tau}-K|\le \eta e^{-\zeta n}, \quad \text{for} \quad \omega\in \Omega_{\tau, \gamma}=\bigcup_{k=1}^{\infty} \left[ \Omega_{\tau k}\cap \Omega_{\gamma k} \right].
\end{equation}
Substituting  $m:=n+\tau$  into \eqref{est:ztau} and  applying \eqref{est:taugamma}, we arrive at part (i),(a). Part  (i),(b) is a corollary of (i),(a).

Choose $n_0$ and set  $\Omega^{[1]} := \Omega_{\tau, \gamma, n_0}$, as in Lemma \ref{lem:estg}. Then
\[
\mathbb P\left(\Omega_{\tau, \gamma, n_0}  \right)>1-\gamma, \quad \tau(\omega)\le n_0, \quad \text{if} \quad \omega \in \Omega_{\tau, \gamma, n_0} .
\]
From \eqref{est:ztau}  we can conclude that, whenever $m\ge n_0$, with a probability exceeding $1-\gamma$, we have the estimate
\[
|z_{m}-K|\le \eta e^{-\zeta (m-n_0)}, \quad m=n_0, n_0+1, n_0+2, \dots,
\]
which concludes the proof.

\section{Appendix B}
 \label{sec:DWC}
 In this section we discuss Directed  Walks Control with changing parameters and present proofs of Theorems \ref{thm:main2} and  \ref{thm:main3}.
 
 
 \subsection{Directed  Walks Control with changing parameters}
 \label{subsec:DWC}

Assume that $z_{n} \in I_{u, K}$, and define  the control that pushes the next $z_{n+1}$ towards ~$K$ by
applying \eqref{def:DWIntr}.
To the best of our knowledge, this method has not been introduced earlier. 
Obviously, $\alpha_j$ should depend on the closeness of $x$ to $K$, to avoid the situation 
when $|x_{n+1}-K| > |x_n-K|$. In a sequence of $\alpha=\alpha_j$ in \eqref{def:DWIntr}, 
we assume that $\alpha_j=a^{-1} \alpha_{j-1}$, $j\in {\mathbb N}$, where $a>1$, 
i.e. $(\alpha_j)$, $j=0,1, \dots, k-1$, is a geometrically decaying finite positive sequence. 

In \eqref{def:DWIntr}, $\zeta_n$ and $\chi_n$  satisfy Assumption \ref{as:chixi}, 
$\ell, \ell_c>0$ are some parameters, which, together with a number $k\in \mathbb N$, finite sequences of numbers $\alpha_j$ and intervals $J^{(j)}$, will be described below.  

We  recall that the purpose of subsequent controls 
with $\alpha_j$ is to bring a solution, which is originally in the interval $I_{u, K}=J^{(0)}$, 
to the interval $I_{\delta, K}=J^{(k)}=I^{(1)}$. 

Assume that there exist $\bar q>0$, $\underline  L>0$,  and $u>0$ such that 
\begin{equation}   
\label{cond:Lqno2}
\bar q<\underline L \quad \mbox{and}\quad  \underline L |z-K|\le |{\rm f} (z)-K| \leq (1+\bar q) |z-K| \quad \mbox{if} \quad z\in I_{u, K}.
\end{equation}
So far \eqref{cond:Lqno2} is the only limitation on $\rm f$ and $u$-neighborhood of $K$. This is equivalent to $\bar q < \underline L
\leq 1+\bar q$, where $1+\bar q$ and $\underline L$ are the upper and the lower bounds of $| {\rm f}(z)-K|/|z-K|$, respectively. In particular,
\eqref{cond:Lqno2} is satisfied if ${\rm f}$ is continuously differentiable and the variation of 
$|{\rm f'}(x)|$ in  $I_{u, K}$ is less than one. 
Lemma~\ref{lem:uCkappa} presents another sufficient condition.

\begin{lemma}
\label{lem:uCkappa}
Let function $\rm f$ satisfy Assumption \ref{as:fu} and condition \eqref {cond:uCkappa} hold.
Then condition \eqref{cond:Lqno2} holds with $\underline L:=1+q-Cu^\kappa$, $\bar q:=q+Cu^\kappa$.
\end{lemma}
\begin{proof}
Note first that, by \eqref{cond:uCkappa}, we have
\[
Cu^\kappa<\frac 12<1+q, \quad \bar q= q+Cu^\kappa<1+q-Cu^\kappa=\underline L.
\]
So $\underline L$ is well defined and satisfies the first condition in \eqref {cond:Lqno2}. To check the second one 
we note that, under Assumption \ref{as:fu}, we have
\[
\sup_{x\in I_{u, K}}|\phi(x-K)|\le C |x-K|^{1+\kappa}\le Cu^\kappa |x-K|,
\]
and thus
\begin{eqnarray*}
&&| {\rm f}(z)-K|\le (1+q+Cu^\kappa) |z-K|=(1+\bar q) |z-K|,\\
&&|{\rm f}(z)-K|\ge (1+q-Cu^\kappa) |z-K|=\underline L |z-K|.
\end{eqnarray*}
\end{proof}

\begin{remark}
\label{rem:1}
Let us note, first, that \eqref{cond:Lqno2} is the only condition required 
for successful implementation of DWC method, and, second,  \eqref{cond:Lqno2} is less restrictive than 
Assumption~\ref{as:fu}: 
for example, ${\rm f} (z)= 1+ 1.4(z-1)-0.2|z-1|$ satisfies \eqref{cond:Lqno2} with $K=1$, $\bar q=0.6$, 
$\underline L=1.2$, while Assumption~\ref{as:fu} is obviously not satisfied. 
\end{remark}

\begin{remark}
\label{rem:2}
If $I_{u,K}$ includes an attractive interval of $\rm f$, the non-controlled sequence enters $I_{u, K}$ at a certain  step: $z_n \in I_{u, K}$ for some $n \in {\mathbb N}$, where $n$ depends on $z_0$.  
If, for example, $\rm f$ is a function with an attractive interval $[z_1,z_2]$ and $K \in (z_1,z_2)$ 
(this includes the case of a stable 2-cycle $\{ z_1,z_2 \}$, $z_1\in (0, K)$, $z_2\in (K,\infty)$ ), and we choose  $u \geq \max\{ 
K-z_1, z_2-K\}$ then, after a finite number of steps, a non-controlled sequence $(z_n )$  enters $I_{u, K}$. 
\end{remark}

Based on \eqref{cond:Lqno2}, choose some 
\begin{equation}   
\label{def:a}
a\in \left(1, \frac {\underline L}{\bar q}\right) \quad \mbox{and denote} \quad \varepsilon_0:=\underline L-\bar qa.
\end{equation}

Define 
\begin{equation}   
\label{def:Kualphaa}
k_{\delta, u}:=\min\left\{j\in {\mathbb N}:  ua^{-(j+1)} <\delta\right\}=\left[\log_a\frac{u}{\delta}\right],
\end{equation}
where $[t]$ is the maximal integer not exceeding $t$,
and a sequence of shrinking neighborhoods of $K$
\begin{equation}   
\label{def:Ika}
I^{(j)}=(K-a^{-j}u, K+a^{-j}u), \quad j = 0, 1, \dots , k_{\delta, u}.
\end{equation}
Thus $k_{\delta, u}$ is the number of the last interval in the sequence $(I^{(j)})$ which strictly includes 
$I_{\delta, K}$, so $I^{(k_{\delta, u}+1)}\subseteq I_{\delta, K}\subset I^{(k_{\delta, u})}$. 
Hence  
for the nested intervals in \eqref{def:DWIntr} we have 
\begin{equation}
\label{def:kJ}
k : = k_{\delta, u}+1, \,\, J^{(0)}:=I_{u, K},\, \, J^{(j)}:=I^{(j)},\,\, \forall \, j=1, \dots,  k_{\delta, u}, \,\,  
J^{(k)}:=I_{\delta, K}.
\end{equation}

Let us define a bound for $\ell_c$ and a constant $B$ satisfying
\begin{equation}   
\label{def:lcBbetasigma}
\ell_c\in \left(0, \, \frac{\varepsilon_0}{2 \bar qa +\varepsilon_0}\right), \quad B\in \left(\frac{\bar qa^2}{\underline L(1-\ell_c)}, \, \frac a{1+\ell_c} \right).
\end{equation}
This allows us to determine bounds for $\ell$ as follows:
\begin{equation}   
\label{def:ell}
\rho:=1+\bar q-\frac{\underline L(1-\ell_c)B}{a^2}, \quad \mu \in (0, 1-\rho), \quad \ell \leq  (1 - \rho -\mu ) \delta.
\end{equation}

\begin{lemma} 
\label{lem:prim2}
Let condition \eqref{cond:Lqno2} hold, $a$ and $\varepsilon_0$ be defined as in \eqref{def:a}. 
Then $\ell_c$, $B$ and $\rho$ are  positive numbers, well defined by  
\eqref{def:lcBbetasigma} and \eqref{def:ell}, respectively.
\end{lemma}

\begin{proof}
Let us note that by \eqref{def:lcBbetasigma}, $\ell_c\in (0,1)$. Further,
$B$ is well defined in \eqref{def:lcBbetasigma} if 
$$
\frac{\bar qa^2}{\underline L(1-\ell_c)}<\frac{a}{(1+\ell_c)},
$$
which is equivalent to
$$
\ell_c < \frac{ \underline L - \bar qa } {\underline L +\bar qa }= \frac{\varepsilon_0}{2 \bar qa +  \varepsilon_0}
$$ 
and is satisfied due to the definition of $\varepsilon_0$ in \eqref{def:a}.
The fact  that 
all other parameters  in \eqref{def:lcBbetasigma}-\eqref{def:ell} are well defined follows from the inequalities
\begin{equation*}
\begin{split}
&1+ \bar{q}-\frac{\underline L(1-\ell_c)B}{a^2}<1+\bar q - \frac{\underline L(1-\ell_c)\frac{\bar qa^2}{\underline L(1-\ell_c)}}{a^2} = 1,\\
&
1+ \bar{q}-\frac{\underline L(1-\ell_c)B}{a^2}>1+\bar q -\frac {\underline L(1-\ell_c)\frac a{1+\ell_c}}{a^2}>1+\bar q -\frac {\underline L}{a}>0.
\end{split}
\end{equation*}
\end{proof}

For each $j =0, 1, \dots,  k_{\delta, u}+1$  we define 
\begin{equation}   
\label{def:contr}
\alpha_j = \underline L a^{-(j+2)}u B,
\end{equation}
where $a$ and $B$ are from \eqref{def:a} and \eqref{def:lcBbetasigma}, respectively.

The following theorem proves that suggested DWC method \eqref{def:DWIntr}, 
after a certain number of steps, brings a solution with $z_0 \in I_{u,K}$ to $I_{\delta,K}$.

\begin{theorem} 
\label{thm:prim2}
Let Assumption \ref{as:chixi} and conditions \eqref{cond:Lqno2}, \eqref {def:lcBbetasigma} hold, 
$\delta\in (0, u)$ and $\mu$ be  defined as in \eqref{def:ell}.
Let $z_n$ be a solution to \eqref{def:DWIntr} with  the initial value $z_0\in I_{u, K}$, 
where 
$J^{(j)}$, $k$ and $\alpha_j$ satisfy  \eqref{def:Ika}, \eqref{def:kJ} and \eqref{def:contr}.   

Then there exists a random moment $\tau$ such that, for any $\omega \in \Omega$, 
\begin{equation}   
\label{est:steps}
z_\tau (\omega)\in I_{\delta,K}, \quad \mbox{and} \quad  \tau(\omega)\le \frac{\ln(\delta/u)}{\ln(1-\mu) } .
\end{equation}
\end{theorem}

\begin{proof} 
If, for some $n\in \mathbb N$ and $j=0, 1, \dots,  k_{u, \delta}$, we have  $z_{n} \in I^{(j)}\setminus I^{(j+1)}$, then
\[
|{\rm f }(z_n)-K|\ge \underline L|z_n-K|\ge \underline L a^{-(j+1)}u, \quad  |{\rm f }(z_n)-K|\le(1+\bar q) |z_n-K|\le(1+\bar q) a^{-j}u.
\]
We have, by  \eqref{def:lcBbetasigma}-\eqref{def:contr},
$\displaystyle \frac{\left|\alpha_j(1+\ell_c \chi_{n+1})\right|}{|{\rm f }(z_n)-K|}
\le \frac{\underline L a^{-(j+2)} u (1+\ell_c)B}{  \underline L a^{-(j+1)}u }=\frac {(1+\ell_c)B}{a}<1
$
\\
and
$\displaystyle 
\frac{\left|\alpha_j(1+\ell_c \chi_{n+1})\right|  }{|{\rm f }(z_n)-K|}\ge \frac{\underline L a^{-(j+2)} u (1-\ell_c)B}{(1+\bar q) a^{-j} u}=\frac {\underline L  (1-\ell_c)B }{a^2(1+\bar q)}$.

Note that 
$\displaystyle 
\frac {\underline L  (1-\ell_c)B }{a^2(1+\bar q)}<\frac {(1-\ell_c)B }{a^2}\le  \frac{(1-\ell_c)}{a(1+\ell_c)} <1$.
Then,\\
$\displaystyle 
\left|1-\frac{\alpha_j(1+\ell_c \chi_{n+1})}{|{\rm f }(z_n)-K|}\right|=1-\frac{\alpha_j(1+\ell_c \chi_{n+1})}{|{\rm f }(z_n)-K|}<1-\frac {\underline L  (1-\ell_c)B }{a^2(1+\bar q)}$.

By \eqref{def:lcBbetasigma} and the above estimates we have
\begin{equation*}
\begin{split}
&\left| z_{n+1}- K \right|=\left|  {\rm f }(z_n) -K+ \ell \zeta_{n+1}- \alpha_j(1+\ell_c \chi_{n+1})\frac{{\rm f }(z_n)-K}{|{\rm f }(z_n)-K|}  \right|
\\
\leq &\left| \left( 1- \frac{ \alpha_j(1+\ell_c \chi_{n+1})}{|{\rm f }(z_n)-K|} \right) ({\rm f }(z_n) -K) \right| + |\ell \zeta_{n+1}|\\
\le &  \left[1-\frac {\underline L (1-\ell_c)B}{a^2(1+\bar q)} \right] (1+\bar q)|z_n -K| + \ell \\
= & \rho  \left| z_{n}-K \right|+\ell\\
 \leq &  \rho  \left| z_{n}-K \right|+(1-\rho-\mu) \delta\\
\leq &  \rho  \left| z_{n}-K \right|+ (1-\rho - \mu)  \left| z_{n}-K \right| \\
= & (1-\mu) |z_{n} - K |.
\end{split}
\end{equation*}
So $z_{n+1}\in I^{(j)}$. If $z_{n+1}\in I^{(j)}\setminus I^{(j+1)}$, then we repeat the above calculations. Since
$\displaystyle 
(1-\mu)^{\bar m} a^{-j}u< a^{-j-1}u
$
for
\begin{equation}
\label{def:barm}
\bar m=\left[-\frac{\ln a}{\ln (1-\mu)}  \right],
\end{equation}
after at most $\bar m$ steps we have  $z_{n+\bar m}\in I^{(j+1)}$. This holds for each $j\le k_{u, \delta}$. 
Therefore, a.s.,  after at most
\[ 
\bar m \times  k(u, \delta)=\left[\frac{\ln u-\ln \delta}{\ln a} \right] \left[-\frac{\ln a}{\ln (1-\mu)}  \right] \le \frac{\ln(\delta/u)}{\ln(1-\mu) } 
\]
 steps  the solution $z_n$ will be in $I_{\delta, K}$.
\end{proof}

\begin{remark}
\label{rem:adaptive}
Theorem \ref{thm:prim2} implies that $n_j$ when switching between $(j-1)$-th and $j$-th control in \eqref{def:DWIntr} occurs, 
are a.s. finite random variables, and $n_j-n_{j-1}\le \bar m$ with $\bar m$ defined by 
\eqref{def:barm}.

So, started in $J^{(0)}$ solution $z_n$ of equation \eqref{def:DWIntr}, for the first time reaches  the interval 
$J^{(k)}\equiv I_{\delta,K}$ at 
a random time $\tau=\tau(u, \delta)$ with an upper estimate as in \eqref{est:steps}. 
After that moment, the new control by noise only (MNC) 
starts working. 
This control was discussed above, in particular, in Appendix A.
\end{remark}

\subsection{Proof of Theorem \ref{thm:main2}} 
\label{subsec:Prmain2}

Theorem \ref{thm:main2}
is a corollary of Theorems~\ref{thm:prim2} and \ref{thm:main1}  and Lemma \ref{lem:uCkappa}.

\subsection{Proof of Theorem \ref{thm:main3}} 
\label{subsec:Prmain3}

\begin{proof}
For simplicity of calculations, we assume $K=0$.

For $x_n\in I_\beta$, we have by Assumption~\ref{as:fu}
\[
|x_{n+1}|\le |1+q+\sigma||x_n|+|\phi(x_n)|\le \bar \Theta \beta+C\beta^{1+\kappa}.
\]
Denoting 
$\displaystyle \beta_1:=\sup \{\beta>0: \bar \Theta \beta+C\beta^{1+\kappa}<u \}$,
we conclude that
\begin{equation}
\label{rel: beta1u}
x_n\in I_{\beta_1} \implies x_{n+1}\in I_{u}.
\end{equation}
\begin{remark}
\label{rem:findbeta1}
Note that $\beta_1$  is easily found in the case $\kappa=1$. Also, assuming  that $\beta<1$ (which is not a restriction!) we can take
$\displaystyle \beta_1<\frac{u}{\bar\Theta +C}$.
\end{remark}
Now we look for $\beta_2$  and a nonrandom  number of steps $\bar s$ such that 
\begin{equation*}
x_n\in I_{\beta_2} \implies x_{n+\bar s}\in I_{\delta}, \quad \mbox{if $\xi_{n+i}\in (1-\iota,1]$, for all $i=1, 2, \dots, \bar s$.}
\end{equation*}
Assume that  $x_n\in I_{\beta_2}$ and $\xi_{n+1}\in (1-\iota,1]$, then
\begin{equation*}
|x_{n+1}|
\le |1+q-(1-\iota)\sigma||x_n|+|\phi(x_n)|\le |\underline\Theta_\iota | |x_n|+C|x_n|^{1+\kappa}
=( |\underline\Theta_\iota |+C\beta_2^\kappa)|x_n|.
\end{equation*}
Since $| \underline \Theta_\iota |<1$, we can take
\begin{equation}
\label{def:beta2H}
\beta_2<\sqrt[\kappa]{\frac{1- |\underline \Theta_\iota |}{C}},\quad H_\iota:= |\underline\Theta _\iota |+C\beta_2^\kappa<1,
 \end{equation}
choose
  \begin{equation}
\label{def:beta_1a}
 \beta <\min\left\{\beta_1, \beta_2\right\}<
  \min\left\{\sqrt[\kappa]{\frac{1- |\underline \Theta_\iota |}{C}},  \,\, \frac{u}{\bar\Theta+C}\right\},
 \end{equation}
 assuming without loss of generality that  $\beta > \delta$,  and set
\begin{equation}
\label{def:bars}
 \bar s:=\left[\frac{\ln \frac \delta{\beta}}{\ln H_\iota} \right]. 
 \end{equation}
 Then, for $i\le \bar s$ on $\Omega_{\iota, n}:=\{\omega\in \Omega: \xi_{n+i}\in (1-\iota,1], \,  i=1, 2, \dots, \bar s\}$, ~
$\displaystyle |x_{n+i}|\le H_\iota^i \beta$.

So, $x_{n+\tau_1}\in I_\delta$, on $\Omega_{\iota, n}$  after at most $\bar s$ steps, where the  random moment $\tau_1 \le \bar s$.

Also, by the  DWC method (see Section \ref{subsec:DWC}), we need  at most $\bar s_1$ steps to reach $I_{\beta}$ from $I_u$.
Let  $ \bar J=\bar s+\bar s_1, \,\,  (a, b]=\left(1-\iota,1 \right].$ By condition \eqref{cond:distr} we have 
$\mathbb P \left\{\xi\in \left(1-\iota,1 \right] \right\}= p_{\iota }>0.$
Applying  now  Lemma \ref{lem:barprob}, we conclude that  there exists an a.s. finite  random moment $\mathcal N$ such that, a.s., 
$\xi_{\mathcal N+i}\in\left( 1- \iota, 1 \right]$, $i=1, \dots, \bar J$.
Consider a solution $x_{\mathcal N}$ at the moment $\mathcal N$. It can be either $x_{\mathcal N}\in I_u\setminus I_{\beta}$ or  $x_{\mathcal N}\in I_{\beta}$.

If  $x_{\mathcal N}\in I_{\beta}$ and since $\xi_{\mathcal N+i}\in\left(1-\iota,1 \right]$ for  $i=1, \dots, \bar s$, we have 
\[
x_{\mathcal N+\tau_1 }\in I_\delta, \quad \mbox{for}\quad \tau_1 \le \bar s.
\]
For $x_{\mathcal N}\in I_u\setminus I_{\beta}$, applying the DWC method, which needs not more than $\bar s_1$ steps to reach $I_\beta$, we have  $x_{\mathcal N+\tau_2}\in I_{\beta}$, $\tau_2\le \bar s_1$.  Since, a.s.,  
\[
\xi_{\mathcal N+\tau_1+i}\in\left(1-\iota,1 \right], \quad \mbox{when} \quad i=1, \dots, \bar s,
\]
we have, a.s.,  
\begin{equation*}
x_{\mathcal N+\tau_1+\tau_2}\in I_\delta.
\end{equation*}
\end{proof}

\end{document}